\documentclass[11pt,a4paper]{amsart}
\usepackage{amsmath,amssymb,amsfonts,amsthm}
\usepackage{a4wide}
\usepackage{color}
\usepackage{pdfsync}
\usepackage{epsfig,epic,eepic,graphicx}
\usepackage{hyperref}
\newtheorem{theorem}{Theorem}
\newtheorem{lemma}{Lemma}
\newtheorem{proposition}{Proposition}

\newtheorem{definition}{Definition}

\newtheorem{remark}{Remark}
\makeatletter

\@addtoreset{equation}{section}
\makeatother

\newcommand{\N}{{\mathbb N}}

\newcommand{\R}{{\mathbb R}}
\newcommand{\C}{{\mathbb C}}
\newcommand{\pa}{{\partial}}
\newcommand{\na}{{\nabla}}

\newcommand{\eps}{{\varepsilon}}

\renewcommand{\a}{\alpha}
\renewcommand{\b}{\beta}
\newcommand{\g}{\gamma}

\renewcommand{\d}{\delta}
\newcommand{\D}{\Delta}

\newcommand{\om}{\omega}
\newcommand{\OM}{\Omega}
\renewcommand{\r}{\rho}

\newcommand{\e}{\varepsilon}
\newcommand{\f}{\varphi}
\newcommand{\F}{\Phi}

\newcommand{\y}{\eta}

\newcommand{\p}{\psi}

\newcommand{\Tc}{\mathcal{T}}
\newcommand{\pd}{\partial}

\newcommand{\loc}{{\rm loc}}

\newcommand{\supp}{{\rm supp}\,}

\newcommand{\cal}[1]{\mathcal{#1}} 

\def\curl{{\rm curl }\,}
\def\div{\hbox{div  }}

\def\capa{\mathrm{cap}}

\title{The 2D Euler equations on singular  domains}
\author[D. G\'erard-Varet \& C. Lacave]{David G\'erard-Varet \& Christophe Lacave. }

\address[D. G\'erard-Varet]{Universit\'e Paris-Diderot (Paris 7)\\
Institut de Math\'ematiques de Jussieu\\
UMR 7586 - CNRS\\
175 rue du Chevaleret\\
75013 Paris\\
France} \email{gerard-varet@math.jussieu.fr}

\address[C. Lacave]{Universit\'e Paris-Diderot (Paris 7)\\
Institut de Math\'ematiques de Jussieu\\
UMR 7586 - CNRS\\
175 rue du Chevaleret\\
75013 Paris\\
France} \email{lacave@math.jussieu.fr}

\begin{document}
\maketitle
\begin{abstract}
We establish the existence of  global weak solutions of the 2D incompressible Euler equations, for a large class of non-smooth open sets. Losely, these open sets  are  the complements (in a simply connected domain)  of a finite number of obstacles with positive Sobolev  capacity. Existence of  weak solutions with $L^p$ vorticity is deduced from a property of domain continuity for the Euler equations, that relates to the so-called $\gamma$-convergence of open sets. Our results complete those  obtained  for convex domains in \cite{taylor}, or for domains with asymptotically small holes \cite{ift_lop_euler,milton}.
\end{abstract}

\tableofcontents

\section{Introduction}
Our concern in this paper is the existence theory for the 2D incompressible Euler flow: for $\Omega$ an open subset of  $\R^2$, we consider the equations
\begin{equation} \label{Euler}
\left\{
\begin{aligned}
 \pa_t u + u \cdot \na u + \na p & = 0, \quad  t > 0, x \in \Omega  \\
 \div u & = 0, \quad  t > 0, x \in \Omega 
\end{aligned}
\right.
\end{equation}
endowed with an initial condition and an impermeability condition at the boundary $\pa \Omega$: 
\begin{equation} \label{conditions}
u\vert_{t=0} = u^0, \quad u \cdot \nu \vert_{\pa \Omega} = 0,   
\end{equation}
where $\nu$ denotes the outward unit normal vector at $\pa \Omega$. Note that it is well-defined  only if $\Omega$ is smooth: later on, in the study of non-smooth open sets,  we shall introduce a weaker form of the impermeability condition.  As usual, $u(t,x)= (u_1(t,x_1,x_2), u_2(t,x_1,x_2))$ and $p = p(t,x_1,x_2)$
 denote the velocity and pressure fields, and the vorticity 
$$   \curl u \: := \: \pa_1 u_2 - \pa_2 u_1 $$   
plays a crucial role in their dynamics.

\medskip
The well-posedness of system \eqref{Euler}-\eqref{conditions}  has been of course the matter of many works, starting from the seminal paper of Wolibner for smooth data in bounded domains   \cite{wolibner}. For the case of smooth data in the whole plane, respectively in exterior domains, see \cite{MG}, resp. \cite{kiku}.   In the case where the vorticity is only assumed to be bounded, existence and uniqueness of a weak solution was established by Yudovich in \cite{yudo}. We quote that the well-posedness result of Yudovich applies to smooth  bounded domains, and to unbounded ones under further decay assumptions. Since the work of Yudovich, the theory of weak solutions has been considerably improved, accounting for vorticities that are only  in $L^1 \cap L^p$ (see the work of Di Perna and Majda \cite{DiPernaMajda}), or that are  positive Radon measures in $H^{-1}$ ({\it cf } the paper of Delort \cite{delort}). We refer to the textbook \cite{majda} for extensive discussion and bibliography.

\medskip
{ A common point in all above studies is that  $\pa \Omega$ is at least  $C^{1,1}$. }Roughly,  the reason is the following: due to the non-local character of the Euler equations, these works rely on   global in space estimates of $u$ in terms of  $\curl u$. These estimates  {\em up to the boundary} involve  kernels of the type Biot-Savart, corresponding to operators such as $\na \Delta^{-1}$.  Unfortunately,  such operators are known  to behave badly in general non-smooth domains.  This explains why well-posedness results are dedicated to regular domains, with a few exceptions. 
  
\medskip
Among those exceptions, one can mention the work \cite{taylor} of M. Taylor related to  convex domains $\Omega$.  Indeed, it is well known that if $\Omega$ is convex,  the solution $\psi$ of the  Dirichlet problem 
$$ \Delta \psi = f \: \mbox{ in } \: \Omega, \quad \psi\vert_{\pa \Omega} = 0 $$
belongs to $H^2(\Omega)$ when the  source term $f$ belongs to $L^2(\Omega)$, no matter the regularity of the domain. Pondering on this remark, Taylor was able to  prove in \cite{taylor} the {\em existence of global weak solutions  in  bounded convex domains}. Nevertheless, this interesting result still leaves aside many situations of practical interest, notably flows around irregular obstacles.

\medskip
The special case of a flow outside a curve has been partly studied in a recent paper by the second author: \cite{lac_euler}. That paper yields the existence of Yudovich-type solutions of the Euler equations in the exterior of a smooth Jordan arc. However, that work relies heavily on the Joukowski transform, and can not be extended easily to more general domains. 

{\em Our ambition in the present paper is to recover the  existence of weak solutions with $L^1 \cap L^p$ vorticity, for a large class of non-smooth open sets. To do so, we will establish a general property of domain continuity  for the Euler equations. }

\medskip
The first part of the paper is devoted to  bounded sets. These sets $\Omega$ are obtained by inserting  a finite  number of obstacles ${\cal C}^1,\dots, {\cal C}^k$ into a simply connected domain $\widetilde \Omega$. More precisely, they can be written as
 \begin{equation} \label{omegatype1}
\Omega\: := \: \widetilde \Omega \: \setminus \: \left( \bigcup_{i=1}^k {\cal C}^i\right), 
\quad k \in \N 
\end{equation}
with the following assumptions
\begin{description}
{ \item[(H1) (connectedness)] $\widetilde \Omega$ is a bounded simply connected domain, ${\cal C}^1, \dots, {\cal C}^k$ are disjoint connected compact subsets of  $\widetilde \Omega$. 
 \item[(H2) (capacity)] For all $i=1\dots k$, $\capa({\cal C}^i) > 0$, where $\capa$ denotes the Sobolev $H^1$ capacity. }
\end{description}

Reminders on  the notion of  capacity are provided in Appendix \ref{app_capacity}. In particular, our assumptions allow to handle  flows around obstacles of positive Lebesgue measure, as well as flows around  Jordan arcs or  curves. They do not cover the case of point obstacles, which have zero capacity.  Let us  insist that no regularity is assumed on $\Omega$: exotic geometries, such as the Koch snowflake, can be considered.

\medskip
Within this setting, it is possible to  establish  the existence of   global  weak solutions of the Euler equations with $L^p$ vorticity. More precisely, we consider initial data satisfying 
\begin{equation} \label{initialdatabounded}
u^0 \in L^2(\Omega), \quad  \curl u^0 \in L^p(\Omega), \quad \div u^0 = 0, \quad u^0 \cdot \nu \vert_{\pa \Omega} = 0,   
\end{equation}
for some  $p \in ]1,\infty]$. Note that, due to the irregularity of  $\Omega$, the condition  $u^0 \cdot \nu \vert_{\pa \Omega} = 0$ has to be understood in a weak sense:
\begin{equation}\label{imperm}
 \int_\Omega u^0 \cdot h = 0  \quad \text{for all } h \in G(\Omega):=\{w\in L^2(\Omega) \ : \ w=\nabla p, \ \text{ for some } p\in H^1_{\loc}(\Omega)\}.
\end{equation}
For any open set $\Omega \subset \R^2$,  such a condition on $u^0 \in L^2(\Omega)$ is equivalent to\begin{equation}\label{impermbis}
u^0 \in {\cal H}(\Omega) \quad \text{where}\quad {\cal H}(\Omega):=\text{completion of } \{ \varphi \in{\cal D}(\Omega)\ | \ \div \varphi =0 \} \text{ in the norm of } L^2.
\end{equation}
This equivalence can be found in \cite[Lemma III.2.1]{Galdi}. Moreover, after proving this lemma, Galdi remarks that if $\Omega$ is a regular bounded or exterior domain, and $u^0$ is a sufficient smooth function, then $u^0$ verifies \eqref{imperm} if and only if $\div u^0 = 0$ and $u^0 \cdot \nu \vert_{\pa \Omega} = 0$.

{ Let us stress that the set of initial data satisfying   \eqref{initialdatabounded}  is large:  we will show later that  for any function $\om^0 \in L^p(\Omega)$, there exists $u^0$ verifying \eqref{initialdatabounded} and $\curl u^0 = \om^0$.}

Similarly to \eqref{imperm}, the weak form of the  divergence free and tangency conditions on the Euler solution $u$  will read:  
\begin{equation} \label{imperm2}
\forall h \in {\cal D}\left([0,+\infty); G(\Omega)\right), \quad \int_{\R^+} \int_\Omega u  \cdot  h = 0.
\end{equation}

Finally, the weak form  of the momentum equation on $u$ will read:
\begin{equation} \label{Eulerweak}
\mbox{for all } \, \varphi \in {\cal D}\left([0, +\infty) \times \Omega\right) \mbox{ with } \div \varphi = 0, \quad  \int_0^{\infty} \int_\Omega \left( u \cdot \pd_t \varphi +   (u \otimes u) : \na \varphi \right)  = -\int_\Omega u^0 \cdot \varphi(0, \cdot).
\end{equation}
{ Our first  main theorem is }
\begin{theorem} \label{theorem1}
Assume that $\Omega$ is of type  \eqref{omegatype1}, with  \mbox{(\textrm{H}1)-(\textrm{H}2)}. Let $p \in (1,\infty]$ and $u^0$ as in \eqref{initialdatabounded}-\eqref{imperm}. Then there exists  
$$ u  \in L^\infty(\R^+; L^2(\Omega)), \quad \mbox{ with } \:  \curl u \in L^\infty(\R^+; L^p(\Omega))$$
which is a global weak solution of \eqref{Euler}-\eqref{conditions} in the sense of \eqref{imperm2} and \eqref{Eulerweak}.

 \end{theorem}
{ In a few words, our existence theorem will  follow from a property of domain continuity for the Euler equations. Namely, we will show that  smooth solutions $u_n$  of the Euler equations in smooth approximate  domains $\Omega_n$ converge to a solution $u$ in $\Omega$. By {\em approximate domains}, we  mean converging to $\Omega$ in the Hausdorff topology. These approximate domains, to be built in Section \ref{section2}, read
$$\Omega_n \: := \:  \widetilde \Omega_n \: \setminus \: \left( \bigcup_{i=1}^k \overline{O_n^i}\right)$$ 
for some smooth Jordan domains $\widetilde \Omega_n$ and $O^i_n$.  A keypoint  is the so-called $\gamma$-convergence of $\Omega_n$ to $\Omega$.  All necessary prerequisites on Hausdorff, resp.   $\gamma$-convergence will be given in Appendix  \ref{app_hausdorff}, resp. Appendix \ref{app_gammaconv}. The compactness argument will be given in Section  \ref{section2} ($p=\infty$) and Section \ref{section5} (finite $p$).}  
Further discussion of domain continuity for the Euler equations is provided in Section \ref{section6}. Possible extension of Theorem \ref{theorem1} to weaker settings (Delort's solutions) is also discussed there. 

\bigskip
In the second part of the paper, we consider general exterior domains $\Omega$. We assume that $\OM$ is the exterior of a bounded obstacle with positive capacity. It reads 
\begin{equation} \label{omegatype2}
 \Omega \: := \: \R^2 \setminus {\cal C}
 \end{equation}
with
\begin{description} 
\item[(H1') (connectedness)] ${\cal C}$ is a connected compact set.
\item[(H2') (capacity)] $\capa({\cal C})  > 0$.
\end{description}

\medskip
Let us point out that to work with square integrable velocities in exterior domains is too restrictive. Therefore, we relax the  condition \eqref{initialdatabounded} on the initial data into
\begin{equation} \label{initialdataunbounded}
u^0 \in L^2_{\loc}(\overline{\Omega}), \quad u^0 \rightarrow 0 \:  \mbox{ as } \: |x| \rightarrow +\infty,  \quad  \curl u^0 \in L^p(\Omega), \quad \div u^0 = 0, \quad u^0 \cdot \nu \vert_{\pa \Omega} = 0,
\end{equation}
where the divergence free and tangency conditions read:
\begin{equation}\label{impermub}
 \int_\Omega u^0 \cdot h = 0  \quad \text{for all } h \in G_{c}(\Omega):=\{w\in L^2_{c}(\overline{\Omega}) \ : \ w=\nabla p, \ \text{ for some } p\in H^1_{\loc}(\Omega)\}.
\end{equation}
Even if we consider only test functions with bounded supports, this condition is still equivalent to the standard impermeability condition when $\Omega$ is regular: if $u^0$ is a sufficiently smooth function, then $u^0$ verifies \eqref{impermub} if and only if $\div u^0 = 0$ and $u^0 \cdot \nu \vert_{\pa \Omega} = 0$.

Similarly to \eqref{imperm2}, the weak form of the  divergence free and tangency conditions on the Euler solution $u$  will read:  
\begin{equation} \label{impermub2}
\forall h \in {\cal D}\left([0,+\infty); G_{c}(\Omega)\right), \quad \int_{\R^+} \int_\Omega u  \cdot  h = 0.
\end{equation}

We make the additional assumption  that 
\begin{equation} \label{initialdataunbounded2} 
 \mbox{ $\curl u^0$ is supported in a compact subset of $\R^2$}
 \end{equation}
 which is classical in this context. We prove in Sections \ref{section3} and \ref{section5} the following result: 
  
\begin{theorem} \label{theorem2}
Assume that $\Omega$ is of type   \eqref{omegatype2}, with (H1')-(H2'). Let $p\in (2,\infty]$ and  $u^0$ satisfying  \eqref{initialdataunbounded}-\eqref{impermub} and \eqref{initialdataunbounded2}. Then, there exists
 $$u \in L^\infty_{\loc}(\R^+; L^2_{\loc} (\overline{\Omega})), \: \mbox{ with } \: \curl u \in   L^\infty(\R^+; L^1\cap L^p (\Omega))$$
which is a global weak solution of \eqref{Euler}-\eqref{conditions} in the sense of \eqref{impermub2} and \eqref{Eulerweak}.
\end{theorem}
Again, the weak solution $u$ is obtained from the compactness of  a sequence of smooth solutions $u_n$ in the approximate  domains $\Omega_n \: := \: \R^2 \setminus \overline{O_n}$. The special case $p=\infty$ will be treated with full details in Section \ref{section3}. The extension to finite $p$ will be sketched in Section \ref{section5}.

\medskip 
Note that our  analysis improves the recent results obtained by the second author  for the exterior of a Jordan arc.   We improve both  the result  in \cite{lac_euler} (existence of solution outside a Jordan arc) and \cite{lac_small} (continuity property for a special class of approximating obstacles $\overline{O_n}$):  we treat more shapes  than just $C^2$ Jordan arcs, and  our convergence  of $\Omega_n$ to $\Omega$  is  expressed through the  Hausdorff distance, which is more general and simple than the conditions in \cite{lac_small}. Therein, one needs stringent convergence properties of the biholomorphisms  that map  $\Omega_n$ to the set $\{|z| > 1\}$. In particular, to obtain the uniform convergence of the first derivatives requires the convergence of the tangent angles of $\pd O_n$.  We refer to \cite{lac_small} for detailed statements. 

We point out  that the limit dynamics in Theorem \ref{theorem2} is expressed differently than in \cite{lac_euler}. Indeed, in this article, extensions $\tilde u_n $ of $u_n$ to the whole plane are considered, resulting in a modified Euler system in the whole plane at the limit. This system is expressed in vorticity form, and reads
\begin{equation} \label{vorticityformulation} \pa_t \omega + u \cdot \na  \omega = 0, \quad \omega \: := \:  \curl u \:  - \:  g_\omega \delta_{\cal C}, \quad t > 0, \:  x \in \R^2, 
\end{equation}
with an additional Dirac mass along the curve. The equivalence between \eqref{vorticityformulation} and the standard formulation \eqref{Eulerweak} of our theorem will be discussed in Section \ref{section6}. In particular, it is proved in \cite{lac_euler} that the velocity blows up near the end-points like the inverse of the square root of the distance, which belongs to $L^p_{\loc}$ for $p<4$. Here, we will obtain some uniform estimates of  the velocity in $L^2_{\loc}$ (see \eqref{na psi}-\eqref{L2 harm}) which are in agreement with the former estimates.

The general idea here is to get uniform  $L^2$ estimates for the velocity (i.e. $H^1$ for the Laplace problem). Combined with $L^p$ bounds on the vorticity,  they will allow us  to establish the existence of weak solutions. As we will show, such  uniform  local estimates do not require any assumption on the regularity of the boundary, but they require that the obstacles have a non zero capacity, which means that we cannot treat the material points. Indeed, in the case of an obstacle which shrinks to a point $P$, Iftimie, Lopes-Filho and Nussenzveig-Lopes in \cite{ift_lop_euler} (one obstacle in $\R^2$) and Lopes-Filho in \cite{milton} (several obstacles in a bounded domain) proved that one has at the limit a term like $x^\perp/(2\pi |x|^2)$ centered at the point $P$, which is not square integrable. Therefore, the assumption of positive capacity appears to be necessary for a $L^2$ approach, and the goal of this article is to prove that it is sufficient to establish the existence. Note that the connection between $H^1$ capacity and the detection of  obstacles is well-known in elliptic theory, notably with regards to the Signorini problem. 

Let us finally insist that even for Yudovich type solutions ($p=\infty$), we  only  deal with global existence, not  uniqueness. The classical proof of uniqueness  requires accurate velocity estimates  in $W^{1,p}(\OM)$ for any $p<\infty$, whereas we only need $L^2$ estimates for the existence theory. In general domains, the Calderon-Zygmund inequality does not hold and the second author considers in \cite{lac_uni} domains with some corners, for which the velocity is proved not to belong to $W^{1,p}$ for all $p$ (precisely, if there is a corner of angle $\alpha>\pi$, then the velocity is no longer bounded in $L^p\cap W^{1,q}$, $p>p_\alpha$, $q>q_\alpha$ with $p_\alpha\to 4$ and $q_\alpha\to 4/3$ as $\alpha\to 2\pi$). Therefore, proving uniqueness seems challenging:  it is only obtained for  compactly supported initial vorticity, with definite sign, when $\OM$ is the interior or the exterior of a simply connected set (see \cite{lac_uni} for details).

\medskip
{\em N.B.} Until the end of the paper, the word {\em domain} will refer to a {\em connected open set}.

\section{Theorem \ref{theorem1} for $p=\infty$} \label{section2}

This section is devoted to the proof of Theorem \ref{theorem1} (bounded domains), in the special case $p=\infty$, that is starting from data with bounded vorticity. The baseline of the proof is 
\begin{itemize}
\item To construct smooth approximations $\Omega_n$ of $\Omega$ (in Hausdorff topology), and smooth approximations $u^0_n$ of  the initial data $u^0$. These smooth data will generate  global smooth solutions $u_n$ of the Euler equation.
\item To obtain uniform bounds on $u_n$, that  provide a converging subsequence. 
\item To show that the limit $u$ of this subsequence satisfies the Euler system in $\Omega$ with initial data $u^0$. The main point is to check that the tangency condition, and the  nonlinear equation on momentum are preserved in the limit $n \rightarrow +\infty$.  
\end{itemize}
However, the reasoning turns out to be not so linear. For instance, the smoothing  of the initial data is not obvious. Standard approximation $u^0_n$   by truncation and convolution of $u^0$ does not work. Indeed, we want to preserve the bounds of $u^0$, that is we want  $(u^0_n)$ to be uniformly bounded in $L^2$, with $\curl u^0_n$  uniformly bounded in $L^\infty$. Consequently, no approximation $u^0_n$ with compact support is allowed: by the uniform bounds just mentioned, the limit would belong to $H^1_0(\Omega)$, which is not convenient to deal with any $u^0$ tangent  to the boundary. The right thing to do is to approximate the initial vorticity $\omega^0$ by some smooth compactly supported $\omega^0_n$, and  to reconstruct an initial velocity $u^0_n$ tangent to $\pa \Omega_n$ through Hodge-De Rham theorem.   
The construction of $\Omega_n$ and $u^0_n$ is detailed in paragraph \ref{subsectionapprox}. 

Nevertheless, uniform bounds on $u^0_n$ can not be established rightaway. They are derived simultaneously to those on $u_n$: roughly, one can say that uniform bounds for $t=0$ and $t > 0$ are equally hard to obtain. They rely on a Hodge decomposition of the velocity field, into its so-called harmonic and rotational parts. Uniform bounds on the harmonic part are established in paragraph \ref{subsectionharm}, those on the rotational part are contained in paragraph  \ref{subsecrot}. Convergence of $u^0_n$ to $u^0$ and continuity of the tangency condition and of the Euler momentum equation follow in paragraph \ref{subsecconclu}. A key ingredient in these paragraphs is the $\gamma$-convergence of $\Omega_n$ to $\Omega$, related to the  domain continuity of the inverse Dirichlet Laplacian. Laplace equations are naturally involved through the use of streamfunctions.

\subsection{Regularization of the data} \label{subsectionapprox} \

 Our starting point is the approximation of $\Omega$ by  smooth domains $\Omega_n$.   We state 
\begin{proposition}  \label{approxlemma}
 Let $\Omega$ of type \eqref{omegatype1}, satisfying (H1). Then, $\Omega$ is the Hausdorff limit of a sequence 
 $$\Omega_n \: := \:  \widetilde \Omega_n \: \setminus \: \left( \bigcup_{i=1}^k \overline{O_n^i}\right),$$
where  $\widetilde \Omega_n$ and   the $O_n^i$'s  are smooth Jordan domains, such that  $\widetilde \Omega_n$, resp.   $\overline{O_n^i}$,   converges  in the Hausdorff sense to  $\widetilde \Omega$, resp.   ${\cal C}^i$. 
\end{proposition}

\begin{proof} Let $U$ be  a bounded simply connected domain. By the Riemann mapping theorem, there exists a unique biholomorphism   $\displaystyle \mathcal{T} : \{ |z| < 1 \} \: \mapsto \:  \widetilde \Omega$ satisfying $\mathcal{T}'(0) > 0$. The sets 
$\displaystyle U_n \: := \:  \mathcal{T}(\{ |z| < 1-1/n \})$  are  smooth Jordan domains, with $(U_n)$, resp. $\overline{U_n}$ converging respectively to $U$  and $\overline{U}$. In particular, applying this argument with $U = \tilde \Omega$ yields the sequence $(\widetilde \Omega_n)$ from the lemma. Also,  if an obstacle ${\cal C}^i$  is the  closure $\overline{O^i}$  of a bounded simply connected domain, the same process provides an approximation of ${\cal C}^i$ by the closure of a smooth Jordan domain $O^i_n$.  

To conclude the proof, it remains to show that any connected compact set ${\cal C}$ can be approximated (in arbitrary short Hausdorff distance) by the closure of  a bounded simply connected domain. Clearly, ${\cal C}$ can be approximated by the closure of a polygonal domain $P$, for instance a non-disjoint and finite union of open squares (the union is not disjoint thanks to the connectedness of ${\cal C}$). Then, the whole point is  to approximate  $P$ by a polygonal domain $P'$ satisfying
$$ \mbox{ number of connected components of } P'^c  \:  = \:    (\mbox{number of connected components of } \: P^c) \: - \: 1.  $$
Indeed, reiteration of such approximation will yield  after a finite number of steps an approximation of $P$ by a simply connected domain. 

To build $P'$, the idea is to draw a polygonal line $L$ connecting the unbounded component of $P^c$ to its nearest bounded connected component.    
One can check that removing such line to $P$ does not break its connectedness. Then, one can thicken the line, that is consider its $\eps$-neighborhood $L^\eps$ and take $P' = P \setminus L^\eps$, which in the limit $\eps \rightarrow 0$ provides an approximation of $P$ arbitrary close in Hausdorff topology.  For the sake of brevity, we leave the details to the reader. 
\end{proof}

\medskip
After smoothing of $\Omega$ into $\Omega_n$, we  need to  smoothen the initial data $u^0$,  to  generate eventually some strong Euler solution in $\Omega_n$. We proceed as follows. Let $\omega^0 := \curl u^0$. By truncation and convolution, there exists some sequence $\omega^0_n \in C^\infty_c(\Omega)$ such that 
$$ \omega^0_n \rightarrow \omega^0 \quad   \mbox{ strongly in }  L^q(\Omega), \quad \forall q \in [1,\infty) $$
and
$$  \| \omega^0_n \|_{L^q} \le  \| \omega^0 \|_{L^q}, \quad \forall q \in [1,\infty]. $$
As $(\Omega_n)_{n \in \N}$ converges to $\Omega$ in the Hausdorff sense, it follows from Proposition \ref{prop4} that $\omega^0_n \in C^\infty_c(\Omega_{N_n})$  for $N_{n}$ large enough. Hence, up to extract a subsequence from $(\Omega_n)_{n \in \N}$ one can assume
$$ \omega^0_n \in  C^\infty_c(\Omega_{n}), \quad \forall n \in \N. $$
To uniquely determine a velocity field $u^0_n$ from $\omega^0_n$, we still need to specify the circulation around each obstacle, in a weak sense. 

First, we introduce some cutoff functions.
For all $i=1\dots k$ and $\eps > 0$, let ${\cal C}^{i,\eps} \: := \:  \{x, \: d(x,{\cal C}^i) \le \eps\}$ the $\eps$-neighborhood of ${\cal C}^i$. Let $\chi^{i,\eps} \in C^\infty_c(\R^2)$  smooth functions satisfying 
 \begin{equation}\label{cutoff}
\chi^{i,\eps} \equiv  1 \: \mbox{ on ${\cal C}^{i,\eps}$}, \quad \chi^{i,\eps} \equiv  0 \: \mbox{ on $\R^2 \setminus  {\cal C}^{i,2\eps}$}.
\end{equation}
By Assumptions (H1)-(H2), there exists $\eps > 0$ and $n_0 = n_0(\eps)$ such that  
$$ \chi^{i,\eps} \equiv  1 \: \mbox{ on } \: \overline{O^i_n}, \quad \chi^{i,\eps} \equiv  0 \: \mbox{ on } \: \overline{O^j_n}, \quad  j \neq i, \quad \chi^{i,\eps} \equiv  0 \: \mbox{ on } \:  \pa \widetilde \Omega_n, \quad \mbox{ for all } \: n \ge n_0.  $$
For brevity, we drop the upperscript $\eps$.

Then, we define the weak circulation of $u^0$ around ${\cal C}^j$
\begin{equation}\label{weak circulation}
\g^j=\g^j (u^0) := -\int_{\OM} \chi^j\curl u^0  -\int_{\OM} u^0 \cdot \na^\perp \chi^j.
\end{equation}

By standard results  related to the Hodge-De Rham theorem, there exists a unique
field $u^0_n \in C^\infty_c(\overline{\Omega_n})$   satisfying 
$$ \curl u^0_n = \omega^0_n, \quad \div u^0_n = 0, \quad u^0_n \cdot \nu \vert_{\pa \Omega_n} = 0, \quad \int_{\pd O^i_n} u^0_n \cdot \tau ds  = \g^i, $$
where $\tau$ denotes the unit tangent vector rotating counterclockwise.

\medskip
We take $(u^0_n)$ as our sequence of initial data. We shall postpone   the convergence of $u^0_n$ to $u^0$ to the end of the section. We consider  for all $n$  the unique smooth solution $u_n$ of the Euler equations in $\Omega_n$, such that 
$$ u_n \cdot \nu  \vert_{\pa \Omega_n} = 0, \quad u_n\vert_{t=0} = u^0_n. $$
Again,   from classical results related to the Hodge-De Rham theorem, the divergence-free smooth fields $u_n$ satisfy in  $\Omega_n$ 
\begin{equation} \label{hodge}
 u_n(t,x) \: = \: \na^\bot \psi^0_n(t,x) \: + \: \sum_{i=1}^k \alpha^i_n(t) \na^\bot \psi^i_n(x)
\end{equation}
where $\psi^0_n$ satisfies the Dirichlet problem 
\begin{equation} \label{psi0}
\Delta \psi^0_n \: = \: \omega_n \: :=  \: \curl u_n \: \mbox{ in } \: \Omega_n, \quad \psi^0_n\vert_{\pa \Omega_n} = 0  
\end{equation}
whereas $\psi^i_n$, $i=1\dots k$ are harmonic functions satisfying 
\begin{equation} \label{psii}
\Delta \psi^i_n \: = \: 0 \: \mbox{ in } \: \Omega_n, \quad \frac{\pa \psi^i_n}{\pa \tau} \vert_{\pa \Omega_n} \: = \: 0, \quad \int_{\pa O^j_n} \frac{\pa \psi^i_n}{\pa \nu } = -\delta_{ij}, \quad  \psi^i_n\vert_{\pd \widetilde \OM_n} = 0, 
\end{equation} 
where $\delta_{ij}$ is the Kronecker symbol  and $\nu$ denotes the unit vector pointing outside $\Omega_n$. Note that $\alpha^i_n$, $i=1 \dots k$ only depends on time (the formula will be given in Proposition \ref{prop : alpha}).

We refer to \cite{kiku} and \cite{milton} for all details. The key point  in proving Theorem \ref{theorem1} is to obtain some compactness on $u_n$ through the study of the $\psi^i_n$'s.

\subsection{Study of the harmonic part}\label{subsectionharm} \

We first focus on the harmonic part of $u_n$, that is the sum at the r.h.s. of \eqref{hodge}. Let us  introduce the auxiliary harmonic functions $\phi^i_n$, $\: i=1\dots k$, that satisfy 
\begin{equation} \label{phii}
\Delta \phi^i_n = 0, \quad \phi^i_n\vert_{\pa \widetilde \Omega_n} = 0, \quad \phi^i_n\vert_{\pa O^j_n} = \delta_{ij}, \quad j=1\dots k. 
\end{equation} 

\medskip
Using the cutoff function introduced in \eqref{cutoff}, we notice that the function $\Phi^i_n \: := \: \phi^i_n - \chi^i$ satisfies 
$$ \Delta \Phi^i_n = - \D \chi^i \quad  \mbox{ in } \: \Omega_n, \quad \Phi^i_n\vert_{\pa \Omega_n} = 0. $$
Let $D$ some big open ball containing all the $\Omega_n$'s.  We can use Proposition \ref{prop9}: as $(\Omega_n)_{n \in \N}$ converges to $\Omega$ in the Hausdorff sense and the complement in $D$ of $\Omega_n$ has at most $k+1$ connected components for all $n$, $(\Omega_n)_{n \in \N}$ $\gamma$-converges to $\Omega$.  We deduce that $\Phi^i_n$ converges in $H^1_0(D)$ to the solution $\Phi^i \in  H^1_0(\Omega)$ of
$$ \Delta \Phi^i = -  \D \chi^i \quad \mbox{ in } \: \Omega, \quad \Phi^i\vert_{\pa \Omega} = 0. $$
Setting $\phi^i \: := \: \Phi^i + \chi^i$, we have for  $i=1\dots k$ the convergence of $\phi^i_n$ to $\phi^i$ strongly in $H^1_0(D)$. 
 
 \medskip
 
Any harmonic harmonic function satisfying 
\begin{equation*} 
\Delta \psi \: = \: 0 \: \mbox{ in } \: \Omega_n, \quad \frac{\pa \psi}{\pa \tau} \vert_{\pa \Omega_n} \: = \: 0,  \quad  \psi\vert_{\pd \widetilde \OM_n} = 0, 
\end{equation*} 
can be decomposed on the $\phi^j_n$'s, then there exists some constants such that
\begin{equation} \label{decompo} \psi^i_n \: = \: \sum_{j=1}^k c^{i,j}_n \, \phi^j_n. 
\end{equation}
We just have established the convergence of the $\phi^j_n$'s, let us now turn to the convergence of the constants $c^{i,j}_n$. We take the normal derivative at both sides of \eqref{decompo} 
 and integrate along $\pa O^m_n$, for $m \in \{ 1, \dots k \}$. We obtain thanks to \eqref{psii} and \eqref{phii}: 
 $$ -\delta_{im} \: = \:  \sum_{j=1}^k c^{i,j}_n \int_{\pa O^m_n}\frac{\pa \phi^j_n}{\pa n} \: = \:   \sum_{j=1}^k c^{i,j}_n \int_{\Omega_n}  \na \phi^j_n \cdot \na \phi^m_n. $$ 
Introducing the  $k \times k$ identity matrix $Id$, this last line reads: 
 $$ -Id  \: = \:   C_n \, P_n, \quad \mbox{ with } \: C_n = \left( c^{i,j}_n \right)_{1 \le i,j \le k}, \quad P_n = \left( \int_{\Omega_n} \na \phi_n^i \cdot \na \phi_n^j \right)_{1 \le i,j \le k}.   $$
 Our goal is to show the convergence of $C_n$: it is therefore enough to prove the convergence of $P_n$ to an invertible matrix $P$. But from the previous step, that is the convergence of $\phi^i_n$ to $\phi^i$ in $H^1_0(D)$, we know that $P_n$ converges to  
 $$P := \left( \int_{D} \na \phi^i \cdot \na \phi^j  \right)_{1 \le i,j \le k}.   $$
 The matrix $P$ is selfadjoint and nonnegative: namely, for any vector $\lambda \in \R^k$, 
 $$ P \lambda \cdot \lambda \: = \: \int_{D} | \na \sum_{i=1}^k \lambda_i \, \phi^i |^2 .$$ 
Thus, to prove the invertibility of $P$, it is enough to show that the $\phi^i$'s are linearly independent. 	Assume {\it a contrario} that 
$$   \sum \lambda_i \,   \phi^i = 0 \: \mbox{ almost everywhere, for some non-zero vector $\lambda$}. $$
Up to reindex the functions, one can assume that  $\lambda_1 \neq 0$. We remind that  the functions $\Phi^i := \phi^i - \chi^i$ belong to $H^1_0(\Omega)$ (see above). Thus, there exists a sequence of functions $\tilde \Phi^i_n$ in $C^\infty_c(\Omega)$ converging to $\Phi^i$ in $H^1_0(\Omega)$, $\: i=1...k$. We set $\tilde \phi^i_n \: := \: \tilde \Phi^i_n + \chi^i$, and introduce  
$$  v_n \: := \:  \frac{1}{\lambda_1} \left( \sum_{i=1}^k \lambda_i \, \tilde \phi^i_n \right) \: \rightarrow  \: 0 \quad  \mbox{ in } \: H^1_0(D). $$
Clearly, $v_n = 1$ on a neighborhood of ${\cal C}^1$. It follows that 
$$ \int_{D} | \na v_n |^2 \: \ge \: \capa_D({\cal C}^1), $$
and letting $n$ go to infinity leads to $\capa_D({\cal C}^1) = 0$.  This contradicts Assumption (H3). 

\medskip
Eventually, we obtain that $P$ is invertible, which  yields a uniform  bound on the $c^{i,j}_n$'s, and their convergence (up to  subsequences) to some limit constants $c^{i,j}$.  
From the above lines and from relation \eqref{decompo}, we deduce that  $(\psi^i_n)_{n \in \N}$ converges (up to a subsequence) to 
$\displaystyle \psi^i \: := \: \sum_{j=1}^k c^{i,j} \, \phi^i$ in $H^1_0(D)$, for all $i=1, \dots, k$. 

Moreover, from the definition of $\p^i_n$ \eqref{psii}, we know that the weak circulations verify
\begin{equation} \label{circharmonic}
\g^j(\na^\perp \p^i)= -\int_{\OM} \na^\perp  \p^i \cdot \na^\perp \chi^j = \lim_{n\to \infty} -\int_{\OM_n} \na^\perp  \p^i_n \cdot \na^\perp \chi^j = \d_{ij}.
\end{equation}

\medskip
To completely control the harmonic part of the velocity $u_n$, it remains to show convergence of the time dependent functions $\alpha^i_n, \: $ $\: i=1\dots k$ in  \eqref{hodge}. We shall use the following proposition, to be found in \cite{milton}:
\begin{proposition}\label{prop : alpha}
For all $i=1\dots k, \:$ $\: \alpha^i_n \: = \:  \int_{\Omega_n} \phi^i_n \, \omega_n \, dx \: + \: \int_{\pa O^i_n} u_n \cdot \tau \,  ds$.
\end{proposition}
 By Kelvin's  theorem, the circulation of $u_n$ on each $\pa O^i_n$ is constant in time  
so that
$$\alpha^i_n \: = \:  \int_{\Omega_n} \phi^i_n \, \omega_n \, dx \: + \: \int_{\pa O^i_n} u^0_n \cdot \tau \,  ds.$$

Moreover, we remind that the vorticity $\omega_n$ obeys the transport equation
$$ \pa_t \omega_n + u_n \cdot \na \omega_n = 0$$
so that the  $L^q$ norms are conserved: 
\begin{equation} \label{conservationlp}
\mbox{ $\| \omega_n(t, \cdot) \|_{L^q(\Omega_n)} = \| \omega^0_n \|_{L^q(\Omega_n)} \le  \| \omega^0 \|_{L^q(\Omega)}$, $\ 1 \le q \le \infty$. }
\end{equation}
We now extend   $\omega_n$ by $0$ outside $\Omega_n$ for all $n$, so that  the sequence  $\left( \omega_{n} \right)_{n \in \N}$ is bounded in $L^\infty(\R^+ \times D)$. Up to the extraction of a subsequence,  we deduce  that 
\begin{equation} \label{convergenceomegan}
\omega_n \rightarrow \omega \:  \mbox{ weakly *  in } \: L^\infty(\R^+ \times D). 
\end{equation}
One has easily that 
\begin{equation} \label{omegazero}
 \omega = 0 \quad \mbox{outside }\: \overline{\Omega}.
 \end{equation}

From this convergence, we infer that $\alpha^i_n$ converges weakly* in $L^\infty(\R^+)$ to 
$$
\alpha^i  :=    \:  \int_{\Omega} \phi^i \, \omega \, dx \: + \: \g^i
$$

Unfortunately, we cannot establish rightaway strong convergence of $(\alpha^i_n)$. We need  some uniform $L^\infty L^2$ bounds on $u_n$, to be  obtained in the section below.

\subsection{Study of the rotational part} \label{subsecrot} \ 

A simple energy estimate on \eqref{psi0}   yields 
$$ \| \na \psi^0_n(t,\cdot) \|_{L^2(\Omega_n)}^2 \: \le \:   \| \omega_n(t, \cdot) \|_{L^2(\Omega_n)} \| \psi^0_n(t, \cdot) \|_{L^2(\Omega_n)}, \quad \forall t,n. $$
Extending $\psi^0_n$ by zero outside $\Omega_n$, we can see it as an element of  $H^1_0(D)$. By applying the Poincar\'e inequality on $D$ and \eqref{conservationlp}, we end up  with 
$$ \|  \psi^0_n(t,\cdot) \|_{H^1_0(D)} \: = \:  \|  \psi^0_n(t,\cdot) \|_{H^1_0(\Omega_n)}\: \le \:  C \, \| \omega_n(t, \cdot) \|_{L^2(\Omega_n)} \: \le \:  C'  $$
uniformly in $t,n$, in particular for $t=0$. Combining this bound on $\psi^0_n(0,\cdot)$  with the estimates on $\p^i_n$ and $\a^i_n(0)$,  we obtain that $u^0_n$ is uniformly bounded in $L^2$. Then, the conservation of energy implies that 
$$\| u_n(t) \|_{L^2(\Omega_n)} \: = \: \|  u^0_n \|_{L^2(\Omega_n)} \: \le \: C, \quad \forall t,$$
that is a uniform $L^\infty L^2$ estimate on $u_n$.

\medskip
On one hand,  this estimate implies  the strong convergence of $\a^i_n$  (and completes the analysis of the harmonic part). Indeed, we compute
$$ (\alpha^i_n)' \: = \:  \int_{\Omega_n} \phi^i_n \, \pd_t \omega_n \, dx \: = \:  -\int_{\Omega_n} \phi^i_n \, \div(u_n\om_n) \, dx \: = \:  \int_{\Omega_n} \na \phi^i_n \cdot u_n \omega_n \, dx.$$
Using the uniform $L^2$ bounds on   $\na \phi^i_n$ and $u_n$, we infer that $\alpha^i_n$ is uniformly bounded in $W^{1,\infty}(\R^+)$ which means that the converge holds strongly in $C^0_{\loc}(\R^+)$.

\medskip
On the other hand, this estimate allows a control of the time derivatives of $\psi^0_n$. Indeed, we observe that $\pa_t \psi^0_n$ satisfies 
\begin{equation*}
\Delta \left( \pa_t \psi^0_n \right) \: = \: \pa_t \omega_n \: =  \: - \div \left( u_n \omega_n\right) \: \mbox{ in } \: \Omega_n, \quad \pa_t \psi^0_n\vert_{\pa \Omega_n} = 0.  
\end{equation*}
Using the uniform $L^\infty L^2$ and $L^\infty$ bounds on $u_n$ and $\omega_n$ respectively, we get similarly
$$  \|  \pa_t \psi^0_n(t,\cdot) \|_{H^1_0(D)} \: \le \: C, \quad   \forall t,n. $$ 
From these bounds and standard compactness lemma \cite{Simon2}, there exists $\psi^0 \in W^{1,\infty}(\R^+;  H^1_0(D))$
 such that up to a subsequence: 
$$ \psi^0_{n} \rightarrow \psi^0 \: \mbox{ weakly* in } \: W^{1,\infty}(\R^+; H^1_0(D)) \: \mbox{ and strongly in }  C^0(0,T; L^2(D)), \quad \forall T > 0. $$
 From the weak convergence of $\psi^0_n$ and $\omega_n$, we infer that 
\begin{equation} \label{limitpsi0}
 \Delta \psi^0(t, \cdot) = \omega(t,\cdot) \: \mbox{ in } \: {\cal D}'(\Omega), \: \mbox{ for almost every } \: t 
 \end{equation}
using again that any compact subset of $\Omega$ is included in $\Omega_n$ for $n$ large enough. 
 
 \medskip
As $\Omega_n$ $\gamma$-converges to $\Omega$, we can use Proposition \ref{prop10}: $\psi^0_n(t, \cdot)$ has for every $t$ a subsequence that converges weakly in $H^1_0(D)$ to a limit in $H^1_0(\Omega)$. Thus, for every $t$,  $\psi^0(t,\cdot)$ belongs to $H^1_0(\Omega)$.

\medskip
Finally, let us prove the strong convergence of $\psi^0_n$ to $\psi^0$ in $L^2(0,T; H^1_0(D))$ for all $T >0$. Therefore, 
we go back to the equation  \eqref{psi0}. We compute:
\begin{equation*} 
\int_0^T \int_{D} |\na \psi^0_n|^2 =  \int_0^T \int_{\Omega_n}  |\na \psi^0_n|^2 = -\int_0^T \int_{\Omega_n} \omega_n \,   \psi^0_n = \: - \int_0^T \int_D \omega_n  \,  \psi^0_n \: \rightarrow - \int_0^T \int_D \omega \, \psi^0 
\end{equation*}
As we know from the previous paragraph  that $\psi^0(t,\cdot)$ belongs to $H^1_0(\Omega)$ for every $t$, we can perform an energy estimate on \eqref{limitpsi0} as well. We get 
$$ \int_0^T \int_D |\na \psi^0|^2 =  \int_0^T \int_\Omega |\na \psi^0|^2 = - \int_0^T \int_\Omega \omega \, \psi^0 = - \int_0^T \int_D \omega \, \psi^0  $$
Hence, 
$$ \int_0^T \int_{D} |\na \psi^0_n|^2  \: \rightarrow \:  \int_0^T \int_D |\na \psi^0|^2  $$
which together with the weak convergence in $W^{1,\infty}(0,T; H^1_0(D))$ yields the strong convergence of $\psi^0_n$ to $\psi^0$ in $L^2(0,T; H^1_0(D))$ for all $T > 0$. 

\begin{remark}
First, we note here that the $\g$-convergence appears crucial to get the good boundary condition $\psi^0(t,\cdot)\in H^1_0(\Omega)$, which has allowed us to get the strong convergence thanks to an integration by parts. Second, we have used $\om^0\in L^\infty$ to get uniform estimates on  $(\alpha^i_n)'$ and $\pa_t \psi^0_n(t,\cdot)$, whereas the other estimates only require $\om^0\in L^p(\Omega)$.
\end{remark}

\subsection{Conclusion of the proof}\ \label{subsecconclu}

We can now conclude the proof of Theorem \ref{theorem1}. Let $(u_n)_{n \in \N}$ be  the sequence of Euler solutions in $\Omega_n$, associated to the initial data $u^0_n$. Each field $u_n$ has the  Hodge decomposition  \eqref{hodge}.  Through  obvious extension   of $\: \psi^m_n, \:$ $\: m=0 \dots k$, it can be seen as an element  of $L^\infty(\R^+; L^2(D))$. By the results of the previous subsections, it  {\em converges strongly in $L^2((0,T) \times D)$  and weakly* in $L^\infty(\R^+; L^2(D))$},  $ \: T > 0, \: $  to the field
$$  u(t,x) \: = \: \na^\bot \psi^0(t,x) \: + \: \sum_{i=1}^k \alpha^i(t) \na^\bot \psi^i(x).$$
Note that  $\psi^0$ belongs to $L^\infty(\R^+; H^1_0(\Omega))$ whereas for $i=1,\dots,k$, $\: \alpha^i$ belongs to $C^0(\R^+)$ and $\psi^i$ belongs to $H^1_0(\widetilde\Omega)$.  Moreover, by construction, one has 
$\displaystyle \curl u  \: = \:  \Delta \psi^0 \: =  \: \omega \in L^\infty(\R^+ \times \Omega)$
as well as the divergence-free and tangency conditions. Indeed, we can decompose $u$ as:
$$  u(t,x) \: = \: \Bigl[ \na^\bot \psi^0(t,x) \: + \: \sum_{i=1}^k \sum_{j=1}^k  \alpha^i(t) c^{i,j}  \na^\bot \Phi^j(x) \Bigl] + \: \sum_{i=1}^k \sum_{j=1}^k  \alpha^i(t) c^{i,j}  \na^\bot \chi^j(x),$$
where $\psi^0(t,x)$ and $\Phi^j(x)$ belong to $H^1_{0}(\Omega)$ (hence the perpendicular gradient verifies \eqref{impermbis}), and it is clear that $ \na^\bot \chi^j$ verifies \eqref{imperm}.

\medskip
An important remark is that all the reasoning we have made so far also applies to the initial data (without the difficulties linked to time dependence). In particular, it can be seen  that  the sequence $(u^0_n)$ converges strongly in $L^2(\Omega)$ (up to a subsequence). Moreover, its limit $\tilde{u}^0$ has a Hodge decomposition,  
$$  \tilde{u}^0(x) \: = \: \na^\bot \psi^{0,0}(x) \: + \: \sum_{i=1}^k \alpha^{0,i} \na^\bot \psi^i(x).$$
 with $\psi^{0,0} \in H^1_0(\Omega)$, $\Delta \psi^{0,0} = \omega^0$   and $ \alpha^{0,i} :=   \int_{\Omega} \phi^i \, \omega^0 \, dx \:  + \g^i$. In particular, it satisfies 
$$\curl \tilde{u}^0  \: = \:  \omega^0,$$
as well as the divergence-free and tangency conditions  \eqref{imperm}. 

Noting that $\chi^j = \phi^j - \Phi^j$ with $\Phi^j\in H^1_0(\OM)$, we compute the weak circulation of $\na^\bot \psi^{0,0}$:
\begin{eqnarray*}
\g^j(\na^\bot \psi^{0,0}) &=& -\int_{\OM} \om^0  \phi^j  -\int_{\OM} \na^\bot \psi^{0,0} \cdot \na^\perp  \phi^j + \int_{\OM} \om^0\Phi^j + \int_{\OM} \na^\bot \psi^{0,0} \cdot \na^\perp \Phi^j \\
&=& -\int_{\OM} \om^0  \phi^j
\end{eqnarray*}
where we have integrated by parts and used that $\psi^{0,0}, \Phi^j\in H^1_0(\OM)$ and $\D \phi^j=0$.  Moreover, we remind (see \eqref{circharmonic}) that 
$$ \g^j(\na^\bot \psi^{i}) = \delta_{ij}.  $$
Combining everything, it follows that the difference $v^0 := \tilde{u}^0 - u^0$ is curl-free, divergence free, with zero weak circulation around each ${\cal C}^i$ and a tangency condition. Let us show that $v^0 = 0$.  First, it is well-known that the curl-free and divergence-free conditions on $v^0$ amount to  Cauchy-Riemann equations for the function $(x+iy) \mapsto (v^0_1 + i v^0_2)$, which is then holomorphic. In particular, $v^0$ is smooth in $\Omega$. If we now prove that the circulation 
\begin{equation}  \label{zerocircul}
\oint_\Gamma v^0 \cdot \tau = 0
\end{equation}
 for any smooth closed curve $\Gamma$ inside $\Omega$, it will follow that $v^0 = \na p^0$ for some smooth $p^0$ inside $\Omega$. In particular, we will have $v^0 \in G(\Omega) \cap {\cal H}(\Omega) = \{ 0 \}$ (see \eqref{imperm}-\eqref{impermbis}).  
 
Thus, let $\Gamma$ be a smooth closed curve in $\Omega$. Let  $J_\Gamma$ the Jordan domain defined by $\Gamma$. Let $\tilde \chi^j = \tilde \chi^{j, \eps}$, $j=1\dots k$,  cut-off functions near the obstacles, satisfying 
 \begin{equation}
\tilde \chi^j =  1 \: \mbox{ in an $\eps$-neighborhood of } \, C^j,  \quad \tilde \chi^j =  0 \: \mbox{ outside a $2\eps$-neighborhood of } \, C^j.
\end{equation}
For $\eps$ small enough (depending on $\Omega$ and $\Gamma$), the supports of $\tilde \chi^j$ do not intersect each other, $\Gamma$ and $\pa \tilde \Omega$. In particular,  $1 - \sum_{j=1}^k \tilde \chi^j$ is identically one near $\Gamma$, and identically zero in a neighborhood of all obstacles contained in $J_\Gamma$.  We deduce: 
$$  \oint_\Gamma v^0 \cdot \tau = \oint_{\Gamma} v^0 \cdot \tau \, (1 -  \sum_{j=1}^k \tilde \chi^j) = \int_{J_\Gamma \cap \Omega} v^0 \cdot \na^\perp (1 -  \sum_{j=1}^k \tilde \chi^j)  = - \sum_{j, C_j \subset J_\Gamma} \int_\Omega  v^0 \cdot \na^\perp  \tilde \chi^j.$$
 Finally, using relations \eqref{cutoff} -\eqref{weak circulation}  (with other cutoff functions $\chi^j$), 
$$   \int_\Omega  v^0 \cdot \na^\perp  \tilde \chi^j = \int_{\Omega} v^0 \cdot \na^\perp (\tilde \chi^j - \chi^j) = -\int_\Omega \curl v^0 (\tilde \chi^j - \chi^j) = 0 $$
($v^0$ has zero weak circulations and is curl-free).

Hence,  $\tilde{u}^0 - u^0 = 0$. In particular, $u^0_n$ converges to $u^0$ strongly in $L^2$. As a byproduct, we obtain the existence and uniqueness of a Hodge decomposition for data $u^0$ satisfying \eqref{initialdatabounded} in the irregular open set $\Omega$.

\medskip
Finally, let $\varphi \in  {\cal D}\left([0, +\infty) \times \Omega\right)$, with $\div \varphi = 0$. For $n$ large enough, the support of $\varphi$ is included in $\Omega_n$ so that:
\begin{equation*} 
  \int_0^{\infty} \int_\Omega \left( u_n \cdot \pd_t \varphi +   (u_n \otimes u_n) : \na \varphi \right)  = -\int_\Omega u^0_n \cdot \varphi(0, \cdot) 
\end{equation*}
By the  strong $L^2$ convergence of $u_n$ to $u$, and of $u^0_n$ to $u^0$,  it follows that $u$ satisfies the weak form of the Euler equations \eqref{Eulerweak}.

\section{Theorem \ref{theorem2} for $p=\infty$} \label{section3}
This section is devoted to the proof of  Theorem \ref{theorem2} (exterior domains), in the special case $p=\infty$. It shares many features with the previous section, related to bounded domains: the existence follows from regularization (paragraph \ref{subsecregul}), derivation of uniform bounds and convergence. Again, the main point is the study of streamfunctions,  through the use of Hodge decomposition.  A big difference with the bounded case is that we loose {\it a priori} the Poincar\'e inequality, that was involved in the treatment of Laplace equations satisfied by the streamfunctions. Still, a Poincar\'e inequality can be used, thanks to the positive capacity of the obstacle, and the fact that the vorticity is compactly supported. This inequality is established in paragraph \ref{subsecpoincare}. Nevertheless, a new difficulty is to  obtain a uniform bound on the compact support of the approximate vorticities $\omega_n$. This can be achieved through a uniform $L^\infty$ bound on the velocity $u_n$, which is itself derived from an explicit representation of $u_n$ of Biot-Savart type. This is described in paragraph \ref{uniformestimates}. From there, arguments are much similar to those of the bounded case.

\subsection{Regularization of the data} \label{subsecregul} \
Let $\Omega$ of type \eqref{omegatype2}, satisfying (H1'). Following the proof of Proposition \ref{approxlemma} , it can be shown that $\Omega$ is the Hausdorff limit of a sequence 
$$\Omega_n \: := \:  \R^2 \: \setminus \:  \overline{O_n}$$ 
where  $O_n$ is a  smooth Jordan domain, whose closure converges in the Hausdorff sense to ${\cal C}$.

\medskip

By \eqref{initialdataunbounded2}, there exists $\rho_{0}>0$ such that $\omega^0:=\curl u^0$ is compactly supported in $B(0,\rho_{0})$. 
As in the previous section, we infer by truncation and convolution that there exists some sequence $\omega^0_n \in C^\infty_c(\Omega_{n}\cap B(0,\rho_{0}))$ such that 
$$ \omega^0_n \rightarrow \omega^0 \quad   \mbox{ strongly in }  L^q(\Omega), \quad \forall q \in [1,\infty) $$
and
$$  \| \omega^0_n \|_{L^q} \le  \| \omega^0 \|_{L^q}, \quad \forall q \in [1,\infty]. $$

\medskip
To build up an appropriate initial velocity $u^0_n$ in $\OM_n := \R^2 \setminus \overline{O_n}$ from the vorticity, we need to specify the value of the circulation somewhere. 
 Let $J$ be a smooth closed Jordan curve in $\OM$ such that  $\mathcal{C}$ is included in the interior of  $J$. For any $n$, we consider as an  initial velocity $u_n^0$  the unique vector field in $\OM_n$ which verifies
\[\div u_n^0 = 0,\quad  \curl u_n^0=\om^0_n, \quad  u_n^0\vert_{\pa O_n}\cdot \nu  = 0,\quad   \int_{J} u_n^0 \cdot \tau \, ds=  \int_{J} u^0 \cdot \tau \, ds \: \text{ and } \lim_{|x|\to \infty} u_n^0(x) = 0.\]
As $u^0$ verifies \eqref{initialdataunbounded}, we note that $u^0$ belongs to  $W^{1,q}_{\loc}(\OM)$ for all finite $q$, hence the quantity $\int_{J} u^0 \cdot \tau \, ds$  is well-defined in the trace sense. 
As $\OM_n$ is smooth, $u^0_n$ generates a unique global strong solution of the Euler equations (see e.g. \cite{kiku}). The transport equation governing the vorticity implies that the $L^q$ norms are conserved:
\begin{equation}\label{norm om}
 \| \om_n(t,\cdot)\|_{L^q(\OM_n)} = \| \om^0_n \|_{L^q} \le \| \omega^0 \|_{L^q}, \ 1\leq q \leq +\infty.
\end{equation}

\medskip
As in the previous part, the Hodge-decomposition will be useful to obtain estimates on the velocity:
\begin{equation} \label{hodge-ext}
 u_n(t,x) \: = \: \na^\bot \psi^0_n(t,x) \: + \: \alpha_n(t) \na^\bot \psi_n(x)
\end{equation}
where $\psi^0_n$ satisfies for any $t$ the Dirichlet problem 
\begin{equation} \label{psi0-ext}
\Delta \psi^0_n  \: = \: \om_n   \: \mbox{ in } \: \Omega_n, \quad \psi^0_n \vert_{\pa \Omega_n} = 0, \quad   \psi^0_n (x) = \mathcal{O}\big(\frac1{|x|}\big) \text{ as } x\to \infty,
\end{equation}
whereas $\psi_n$ is the harmonic function satisfying 
\begin{equation} \label{psii-ext}
\Delta \psi_n \: = \: 0 \: \mbox{ in } \: \Omega_n, \quad \frac{\pa \psi_n}{\pa \tau} \vert_{\pa \Omega_n} \: = \: 0, \quad \int_{\pa O_n} \frac{\pa \psi_n}{\pa \nu } = -1, \quad    \psi_n(x) = \mathcal{O}\big(\ln |x|\big) \text{ as } x\to \infty.
\end{equation} 
The function $\a_n(t)$ is the sum of the circulation of the velocity around $O_n$ and the mass of the vorticity $\int_{\OM_n} \om_n(t,\cdot)$. By the Kelvin's circulation theorem and the transport nature of the vorticity equation, we infer that these two quantities are conserved, hence
\begin{equation} \label{definitionalpha}
 \a_n(t)\equiv \a_n =  \int_{\pd O_n} u^0_n \cdot \tau \, ds  + \int_{\OM_n} \om^0_n = \int_{\pd J} u^0_n \cdot \tau \, ds - \int_{A_n} \om^0_n  + \int_{\OM_n} \om^0_n
  \end{equation}
  where $A_n := \OM_n \setminus Ext(J)$ (with $Ext(J)$ the exterior of $J$), hence
 \begin{equation}
\label{definitionalpha2}
\a_n \: \rightarrow \: \alpha \: := \:  \int_{J} u^0 \cdot \tau \, ds  + \int_{Ext(J)} \om^0.
\end{equation}

\subsection{Poincar\'e inequality in exterior domains}  \label{subsecpoincare}   \ 

Thanks to the properties in \eqref{psi0-ext}, we integrate by parts to obtain:
\begin{equation}\label{IPP}
 \|  \na  \psi^0_n  \|_{L^2(\OM_n)}^2 \leq \| \om_n \|_{L^2(\OM_n)}  \| \psi^0_n  \|_{L^2(\OM_n\cap \supp \om_n)}.
\end{equation}
In the case of a bounded domain, we used the Poincar\'e inequality on a domain $D$ containing all the $\OM_n$'s. The idea here is to establish a similar inequality, thanks to the $\g$-convergence of $\overline{O_n}$ to $\mathcal{C}$ with $\capa\  \mathcal{C} >0$.

\begin{lemma}\label{poincare}
Let $\r$ be a positive number such that $ {\cal C} \subset B(0,\r)$. Assume $\Omega$ is of type \eqref{omegatype2}, with (H1')-(H2'), then there exists $C_\r>0$ and $N_\r$, depending only on $\r$, such that
\[  \| \f  \|_{L^2(\OM_n\cap B(0,\r))} \leq C_\r \| \na \f  \|_{L^2(\OM_n \cap B(0,\r))}, \ \forall \f \in C^\infty_c (\OM_n),\ \forall n\geq N_\r.\]
\end{lemma}

\begin{proof}
Let us assume that the conclusion is false, which means that for any $k\in \N$, if we choose $C_\r = k$ and $N_\r=\max(k,n_{k-1})$, then there exist $n_k \geq N_\r$ and $\f_k \in C^\infty_c ( \OM_{n_k})$ such that
\[ \| \f_k  \|_{L^2(\OM_{n_k} \cap B(0,\r))} >  k \| \na \f_k  \|_{L^2(\OM_{n_k} \cap B(0,\r))}.\]
Dividing $\f_k$ by $\| \f_k  \|_{L^2(\OM_{n_k} \cap B(0,\r))}$, we can consider that $\| \f_k  \|_{L^2(\OM_{n_k} \cap B(0,\r))}=1$, which implies that $\| \na \f_k  \|_{L^2(\OM_{n_k} \cap B(0,\r))}$ tends to zero as $k$ tends to infinity. Therefore, extracting a subsequence if necessary, we have that
\[ \f_k \to \f \text{ weakly in }H^1(B(0,\r)) \text{ and  strongly in } L^2(B(0,\r)).\]
It follows that $\f$ is a non zero contant on $B(0,\r)$, because $\| \f  \|_{L^2(B(0,\r))}=1$ and $ \na\f_k \rightharpoonup 0$ weakly in $L^2(B(0,\r))$.

We introduce a cutoff function $\chi$ which is equal to zero in $B(0,\r)^c$, and equal to $1$ in some neighborhoods of $O_n$'s and $\mathcal{C}$. Then, 
\begin{equation}\label{eq1}
\chi \f_k \text{ belongs to } H^1_0(B(0,\r) \setminus \overline{O_n}), 
\end{equation}
and 
\begin{equation}\label{eq2}
\chi \f_k \rightharpoonup \chi\f \text{ weakly in }H^1_0(B(0,\r)).
\end{equation}

However, the sequence $(B(0,\r) \setminus \overline{O_n})$ converges to $ B(0,\r) \setminus \mathcal{C}$ in the Hausdorff sense  
and as $\overline{O_n}$ is connected for all $n$, Proposition \ref{prop9}  implies: 
$$\mbox{ $B(0,\r) \setminus \overline{O_n}$ $\gamma$-converges to $ B(0,\r) \setminus \mathcal{C}$.}$$
 Combining \eqref{eq1}, \eqref{eq2} and Proposition \ref{prop10}, we obtain that $\chi\f$ belongs to $H^1_0(B(0,\r) \setminus \mathcal{C})$.

Next, Proposition \ref{prop11} implies that $\chi\f=0$ quasi everywhere in $\mathcal{C}$, i.e. everywhere except on a set of zero capacity. This is in contradiction with $\capa(\mathcal{C})>0$ and the fact that $\chi\f$ is equal to a non zero constant in $\mathcal{C}$. The conclusion of the proof follows.
\end{proof}

We want to apply the previous lemma to \eqref{IPP}, but we remark that an important issue is to control the size of the support of $\om_n$ independently of $n$. As $\om_n$ is transported by $u_n$, we will prove that the velocity is uniformly bounded far from the domains $O_n$.

\subsection{Uniform estimates of the velocity far from the boundaries.} \label{uniformestimates}\ 

The advantage of working outside one simply connected domain is the explicit formula of $\psi^0_n$ and $\psi_n$ in terms of biholomorphisms. 
We denote $D := \{ |z| < 1 \}$ the open unit disk,   $\OM_n:= \R^2\setminus \overline{O_n}$ the approximate exterior domain given by Proposition \ref{approxlemma},  and $\D:= \{ |z| > 1 \}$  the exterior of the closed unit disk.  From the Riemann mapping theorem, it is  easily seen  that there is  a unique  biholomorphism 
$$  \mathcal{T}_n : \Omega_n \mapsto \Delta, \quad \mbox{ with } \:   \mathcal{T}_n(\infty) = \infty, \quad  \mathcal{T}_n'(\infty) > 0. $$
We remind that the last two conditions mean  
$$\mathcal{T}_n(z) \sim \lambda_n z, \quad   |z| \sim +\infty, \quad \mbox{ for some } \:  \lambda_n> 0. $$
With such notations, we have 
\begin{equation}\label{kernel}
\na^\perp \psi^0_n (t,x)= \frac{1}{2\pi} D\Tc_n^T(x) \int_{\OM_n} \Bigl( \frac{\Tc_n(x)-\Tc_n(y)}{|\Tc_n(x)-\Tc_n(y)|^2}-\frac{\Tc_n(x)-\Tc_n(y)^*}{|\Tc_n(x)-\Tc_n(y)^*|^2}\Bigl)^\perp \om_n(t,y)\, dy
\end{equation}
and
\begin{equation}\label{harmonic}
\na^\perp \psi_n (t,x)= \frac{1}{2\pi} D\Tc_n^T(x) \frac{\Tc_n(x)^\perp}{|\Tc_n|^2}
\end{equation}
with the notation $z^*=\dfrac{z}{|z|^2}$ (see e.g. \cite{ift_lop_euler,lac_small} for an introduction to the Biot-Savart law in exterior domains). 

\medskip
Like in \cite{ift_lop_euler,lac_euler,lac_small}, a key point is  to control  $\mathcal{T}_n$ when  $\overline{O_n}$ tends to $\mathcal{C}$.  We state 
\begin{proposition} \label{propconvbiholo}
Let $\Pi$ be the unbounded connected component of $\Omega$. There is a unique biholomorphism $\mathcal{T}$ from $\Pi$ to $\D$, satisfying $\mathcal{T}(\infty)=\infty$, $\mathcal{T}'(\infty) > 0$. Moreover, one has  the following convergence properties:  
\begin{itemize}
\item[i)] $\mathcal{T}_n^{-1}$ converges uniformly locally to $\mathcal{T}^{-1}$ in $\D$. 
\item[ii)] $\mathcal{T}_n$ (resp. $\mathcal{T}_n'$) converges uniformly locally to $\mathcal{T}$ (resp. to $\mathcal{T}'$) in $\Pi$. 
\item[iii)] $|\mathcal{T}_n|$ converges uniformly locally to $1$ in $\Omega\setminus\Pi$.
\end{itemize}
\end{proposition}
This proposition is a consequence of the celebrated theorem of Caratheodory on the convergence of biholomorphisms; we refer to Appendix \ref{app_biholo} for all details and proof. From there, we deduce:

\begin{lemma}\label{est kernel}
Let $R_0$ large enough so that  ${\cal C} \subset B(0,R_0)$ and $p\in (2,\infty]$ fixed. Then, there exists $\displaystyle C_0=C(\|\om^0\|_{L^1}, \|\om^0\|_{L^p},R_0,p)$   such that the function
\begin{equation}\label{kernel bis}
f_n(t,x):= \frac{1}{2\pi} D\Tc^T(x) \int_{\Pi} \Bigl( \frac{\Tc(x)-\Tc(y)}{|\Tc(x)-\Tc(y)|^2}-\frac{\Tc(x)-\Tc(y)^*}{|\Tc(x)-\Tc(y)^*|^2}\Bigl)^\perp \om_n(t,y)\, dy
\end{equation}
verifies
\[ \|f_n(t,x) \|_{L^\infty(\R^+\times B(0,R_0)^c)} \leq C_0,\  \forall n.\]
Moreover, for any compact $K$ outside $\overline{B(0,R_0)}$, there exists $N_K$ such that 
\[ \Big\|\na^\perp \psi^0_n (t,x) \Big\|_{L^\infty(\R^+\times K)} \leq 2C_0+1, \ \forall n\geq N_K,\]
where $\na^\perp \psi^0_n$ is given in \eqref{kernel}.
\end{lemma}

\begin{proof}
Let $\tilde R_0 < R_0$ such that ${\cal C} \subset B(0,\tilde R_0)$. 
We decompose the integral \eqref{kernel bis} into three parts:
\begin{eqnarray*}
f_n(t,x) &=& \frac{1}{2\pi} D\Tc^T(x)\Bigl[ \int_{\Pi \cap B(0,\tilde R_0)}\frac{(\Tc(x)-\Tc(y))^\perp}{|\Tc(x)-\Tc(y)|^2} \om_n(t,y)\, dy\\
&& +  \int_{B(0,\tilde R_0)^c } \frac{(\Tc(x)-\Tc(y))^\perp}{|\Tc(x)-\Tc(y)|^2} \om_n(t,y)\, dy - \int_{\Pi} \frac{(\Tc(x)-\Tc(y)^*)^\perp}{|\Tc(x)-\Tc(y)^*|^2} \om_n(t,y)\, dy\Bigl]\\
&=&  \frac{1}{2\pi} D\Tc^T(x) (\mathcal{I}_1 + \mathcal{I}_2 + \mathcal{I}_3).
\end{eqnarray*}

By the definition of $\Tc$ (see Proposition \ref{propconvbiholo}), there exists some $(\b,\tilde\b)\in \R_*^+\times \C$ such that:
\[ \Tc(z)=\b z+ \tilde \b + \mathcal{O}(\frac1z) \text{ as } z\to \infty.\]
Then there exists $C_1$ such that $\| D\Tc \|_{L^\infty(B(0,R_0)^c)}\leq C_1$. If the boundary $\pd \Pi$ is rough, we recall that such an inequality does not hold in $L^\infty(\Pi)$ (see for instance \cite{Mazya,lac_uni}). This remark underlines the importance of $R_0$.

As $\Tc$ is continuous and one-to-one, there exists $\d>0$ such that 
\[ \mathrm{dist} \Bigl( \Tc(\pd B(0,\tilde R_0)) ; \Tc(\pd B(0, R_0)) \Bigl) \geq \d.\]
Then $|\Tc(x)-\Tc(y)|\geq\d$ for any $x\in B(0, R_0)^c$ and $y\in \Pi \cap B(0,\tilde R_0)$. Hence, for any $x\in B(0, R_0)^c$, we have
\[|\mathcal{I}_1 | \leq \frac{1}{\d} \int_{\Pi \cap B(0,\tilde R_0)} |\om_n(t,y)|\, dy \leq \frac{ \| \om^0\|_{L^1}}{\d},\]
where we have used \eqref{norm om}.

As $|\Tc(y)^*|\leq 1\leq |\Tc(y)|$, we also have $|\Tc(x)-\Tc(y)^*|\geq\d$ for any $x\in B(0, R_0)^c$. Therefore, we obtain
\[|\mathcal{I}_3 | \leq \frac{1}{\d} \int_{\Pi} |\om_n(t,y)|\, dy \: \le \:   \frac{ \| \om^0\|_{L^1}}{\d},\]

Concerning the last part $\mathcal{I}_2$, we introduce $z=\Tc(x)$ and 
\[g(t,\y):= \om_n(t,\Tc^{-1}(\y)) | \det D\Tc^{-1}(\y) | \mathbf{1}_{\Tc (B(0,\tilde R_0)^c)}(\y).\]
Changing variables $\y=\Tc(y)$, we compute
\[\mathcal{I}_2 = \int_{\R^2} \frac{(z-\y)^\perp}{|z-\y|^2} g(t,\y)\, d\y.\]
Changing variables back, we obtain that
\[ \| g(t,\cdot) \|_{L^1(\R^2)} = \| \om_n(t,\cdot) \|_{L^1(B(0,\tilde R_0)^c )} \leq \| \om^0 \|_{L^1}.\]
Using the behavior at infinity of $\Tc^{-1}$, we infer that there exists $C_2$ such that
\[ | \det D\Tc^{-1}(\y) | \leq C_2,\ \forall \y \in \Tc (B(0,\tilde R_0)^c),\]
hence
\[ \| g(t,\cdot) \|_{L^p(\R^2)} \leq C_2^{1-1/p} \| \om_n(t,\cdot) \|_{L^p} \le C_2^{1-1/p} \| \om^0 \|_{L^p}.\]
This last argument explains why we split the integral into several parts: we cannot prove that $ \det D\Tc^{-1}$ is bounded up to the boundary, in particular when its boundary is rough. Using a classical estimate for the Biot-Savart kernel in $\R^2$ (see e.g. \cite[Lemma 3.5]{lac_small}), we write
\[|\mathcal{I}_2| \leq C_3 \| g(t,\cdot) \|_{L^1(\R^2)}^{\a_p}\| g(t,\cdot) \|_{L^p(\R^2)}^{1-\a_p} \leq C_3 C_2^{(1-\a_p)(1-1/p)} \| \om^0 \|_{L^1}^{\a_p}\| \om^0 \|_{L^p}^{1-\a_p},\]
where $C_3=C_3(p)$ is a universal constant, and $\a_p\in (0,1)$ (only true if $p\in (2,\infty]$, for example $\a_\infty=1/2$).

Putting $C_0:= \frac{C_1}{2\pi}\Bigl(\frac{ 2\| \om^0\|_{L^1}}{\d}+C_3 C_2^{(1-\a_p)(1-1/p)} \| \om^0 \|_{L^1}^{\a_p}\| \om^0 \|_{L^p}^{1-\a_p}\Bigl)$, we have established the first inequality:
\[ \|f_n(t,x) \|_{L^\infty(\R^+\times B(0,R_0)^c)} \leq C_0,\  \forall n.\]

\bigskip

We treat now $\na^\perp \psi^0_n$. Let $K$ be a compact set in $B(0,R_0)^c$. Let $\tilde K$ a compact set satisfying 
$$K \subset \tilde K \subset B(0,\tilde R_0)^c, \quad \mbox { and } \: \mathrm{dist} \Bigl( \Tc(\pd K) ; \Tc(\pd \tilde K) \Bigl) \ge  \d. $$
One can take for instance $\tilde K = \{ \tilde R_0\le |z| \le R_1\}$ for $R_1$ large enough. Again, we decompose the integral \eqref{kernel} into three parts:
\begin{eqnarray*}
\na^\perp \psi^0_n (t,x) &=& \frac{1}{2\pi} D\Tc_n^T(x)\Bigl[ \int_{\OM_n \setminus \tilde K}\frac{(\Tc_n(x)-\Tc_n(y))^\perp}{|\Tc_n(x)-\Tc_n(y)|^2} \om_n(t,y)\, dy\\
&& +  \int_{\tilde K} \frac{(\Tc_n(x)-\Tc_n(y))^\perp}{|\Tc_n(x)-\Tc_n(y)|^2} \om_n(t,y)\, dy - \int_{\OM_n} \frac{(\Tc_n(x)-\Tc_n(y)^*)^\perp}{|\Tc_n(x)-\Tc_n(y)^*|^2} \om_n(t,y)\, dy\Bigl]\\
&=&  \frac{1}{2\pi} D\Tc_n^T(x) (\mathcal{J}_1 + \mathcal{J}_2 + \mathcal{J}_3).
\end{eqnarray*}

By the uniform convergence of $D\Tc_n$ to $D\Tc$ in $K$ (see Proposition \ref{propconvbiholo}), for any $\e_1>0$ there exists $N_1$ such that 
\[\| D\Tc_n \|_{L^\infty(K)}\leq C_1+\e_1, \ \forall n\geq N_1 .\]

By the uniform convergence of $\Tc_n$ to $\Tc$ in $\tilde K$, there exists $N_2>0$ such that 
\[ \mathrm{dist} \Bigl( \Tc_n(\pd \tilde K) ; \Tc_n(\pd K) \Bigl) \geq \d/2, \ \forall n\geq N_2.\]
Then $|\Tc_n(x)-\Tc_n(y)|\geq \d/2$ for any $x\in K$ and $y\in \OM_n \setminus \tilde K$. Hence, for any $x\in K$, we have
\[|\mathcal{J}_1 | \leq \frac{2}{\d} \int_{\OM_n \setminus \tilde K} |\om_n(t,y)|\, dy \leq \frac{ 2\| \om^0\|_{L^1}}{\d}, \ \forall n\geq N_2.\]

As $|\Tc_n(y)^*|\leq 1\leq |\Tc_n(y)|$, we also have $|\Tc_n(x)-\Tc_n(y)^*|\geq\d/2$ for any $x\in K$. Therefore, we obtain
\[|\mathcal{J}_3 | \leq \frac{2}{\d} \int_{\OM_n} |\om_n(t,y)|\, dy \: \le \:  \frac{2 \| \om^0\|_{L^1}}{\d},\]

Concerning the last part $\mathcal{J}_2$, we introduce $z=\Tc_n(x)$ and 
\[g_n(t,\y):= \om_n(t,\Tc_n^{-1}(\y)) | \det D\Tc_n^{-1}(\y) | \mathbf{1}_{\Tc_n (\tilde K)}(\y).\]
Changing variables $\y=\Tc_n(y)$, we compute
\[\mathcal{J}_2 = \int_{\R^2} \frac{(z-\y)^\perp}{|z-\y|^2} g_n(t,\y)\, d\y.\]
Changing variables back, we obtain that
\[ \| g_n(t,\cdot) \|_{L^1(\R^2)} = \| \om_n(t,\cdot) \|_{L^1(\tilde K )} \leq \| \om^0 \|_{L^1}.\]
Using the uniform convergence of $D\Tc_n^{-1}$ to $D\Tc^{-1}$ in a compact big enough (such that $\Tc_n (\tilde K)\subset D$), for any $\e_3>0$ there exists $N_3$ such that
\[ | \det D\Tc_n^{-1}(\y) | \leq C_2+\e_3,\ \forall \y \in \Tc_n (\tilde K), \ \forall n\geq N_3\]
hence
\[ \| g_n(t,\cdot) \|_{L^p(\R^2)} \leq (C_2+\e_3)^{1-1/p} \| \om_n(t,\cdot) \|_{L^p} \: \le \:  (C_2+\e_3)^{1-1/p} \| \om^0 \|_{L^p}.\]
Finally, we use the classical estimate for the Biot-Savart kernel in $\R^2$:
\[|\mathcal{J}_2| \leq C_3 \| g_n(t,\cdot) \|_{L^1(\R^2)}^{\a_p}\| g_n(t,\cdot) \|_{L^p(\R^2)}^{1-\a_p} \leq C_3 (C_2+\e_3)^{(1-\a_p)(1-1/p)} \| \om^0 \|_{L^1}^{\a_p}\| \om^0 \|_{L^p}^{1-\a_p}.\]

Choosing well $\e_1$ and $\e_3$, we find $N_K=\max(N_1,N_2,N_3)$ such that
\[ \Big\|\na^\perp \psi^0_n (t,x) \Big\|_{L^\infty(\R^+\times K)} \leq 2C_0+1, \ \forall n\geq N_K,\]
which ends the proof.
\end{proof}

The reason why we divide the proof in two parts is to obtain a constant $C_0$ independent of the compact set $K$. Although $N_K$ depends on $K$, the independence of $C_0$ with respect to $K$ will be crucial for the uniform estimate of the vorticity support. The harmonic part is easier to estimate.

\begin{lemma}\label{est harm}
Let $R_0$ a positive number such that ${\cal C} \subset B(0,R_0)$. Then, there exists $C_0=C(R_0)$ such that the function
\begin{equation}\label{harm bis}
\p(x):= \frac{1}{2\pi} \ln |\Tc(x)|
\end{equation}
verifies
\[  \Big\|\na^\perp \psi (x) \Big\|_{L^\infty(B(0,R_0)^c)} \leq C_0.\]
Moreover, for any compact $K$ outside $\overline{B(0,R_0)}$, there exists $N_K$ such that
\[ \Big\|\na^\perp \psi_n (x) \Big\|_{L^\infty(K)} \leq 2C_0+1, \ \forall n\geq N_K.\]
\end{lemma}

\begin{proof} The first part comes from the behavior of $\Tc$ at infinity:
\[ \Tc(z)=\b z+ \tilde \b + \mathcal{O}(\frac1z) \text{ as } z\to \infty.\]

The second point is a direct consequence of the uniform convergence of $\mathcal{T}_n$ in $K$ (see Proposition \ref{propconvbiholo}).
\end{proof}

\subsection{Support of the vorticity and $H^1$ estimates}\label{sub support}\ 

Let $\r_0$ such that $\cup_n \supp \om^0_n \cup  \supp \om^0 \cup  \Pi^c \subset B(0,\r_0)$ and $p\in (2,\infty]$ fixed.  Let
$$C_0=C_0(\|\om^0\|_{L^1}, \|\om^0\|_{L^p},\r_0,p)$$
the constant of Lemmata \ref{est kernel} and \ref{est harm}. Let $C := (2 C_0 + 1) (2 + |\alpha|)$, where $\alpha$ was defined in \eqref{definitionalpha2}.    We fix a time $T>0$ and we introduce
\[ K_T := \overline{B(0,\r_0+ C T)}\setminus B(0,\r_0).\]
Together with \eqref{hodge-ext}-\eqref{definitionalpha2}, Lemmata \ref{est kernel} and \ref{est harm} provide some $N_T$ such that
\[ \|u_n \|_{L^\infty(\R^+\times K_T)} \leq C, \ \forall n\geq N_T.\]
As $\om_n$ is transported by $u_n$, we can conclude that
\[ \supp \om_n(t,\cdot) \subset B(0,\r_0+ C t ),\ \forall t\in [0,T],\ \forall n\geq N_T.\]

Finally we can complete the estimate of  $\| \na  \psi^0_n  \|_{L^2(\OM_n)}$. Let $\r_T:=\r_0+ C T$, then Lemma \ref{poincare} implies that there exist $C_{\r_T}$ and $N_{\r_T}$ such that 
\[  \| \p_n^0  \|_{L^2(\OM_n\cap \supp \om_n)} \leq C_{\r_T} \| \na \p_n^0 \|_{L^2(\OM_n \cap B(0,\r_T))}, \ \forall t\in [0,T],\ \forall n\geq \max(N_{\r_T},N_T).\]
Combining with \eqref{IPP} and \eqref{norm om}, we obtain:  $\displaystyle \forall t\in [0,T],\ \forall n\geq \max(N_{\r_T},N_T)$:
\begin{equation}\label{na psi}
 \|  \na  \psi^0_n  \|_{L^2(\OM_n)} \leq C_{\r_T} \| \om_n \|_{L^2(\OM_n)} = C_{\r_T} \| \om_n^0 \|_{L^2(\OM_n)} \leq C_{\r_T} \| \om^0 \|_{L^2}.
\end{equation}
Using again Lemma \ref{poincare} on any compact $K$ of $\R^2$, we conclude that there exist $C_{K}$ and $N_{K}$ such that
\begin{equation}\label{psi loc}
 \| \psi^0_n  \|_{L^2(\OM_n\cap K)} \leq C_K   \| \na \p_n^0 \|_{L^2(\OM_n)}   \leq C_K  C_{\r_T} \| \om^0 \|_{L^2}, \ \forall t\in [0,T],\ \forall n\geq \max(N_{\r_T},N_T,N_K),
\end{equation}
where $C_K$ depends on the diameter of $K$. We recall that $\psi^0_n$ is not square integrable at infinity (see \eqref{psi0-ext}), but \eqref{psi loc} will be sufficient to obtain local convergence.  

\medskip
We end this subsection with a $L^2_{\loc}$ estimate of $\na^\perp \p_n$ up to the boundary. Let $R_0>0$ and $\chi$ be a cutoff function equal to $1$ in $B(0,R_0)$ and to $0$ outside $B(0,R_0 +1)$. Then, $\chi \p_n$ verifies
\begin{equation*}
\Delta (\chi\psi_n) \: = \: \tilde\om_n := 2\na\chi\cdot\na \p_n+\p_n\D\chi \: \mbox{ in } \: \Omega_n\cap B(0,R_0+1), \quad \chi \psi_n\: = \: 0  \: \mbox{ on } \: \pd\Omega_n\cup \pd B(0,R_0+1).
\end{equation*} 
Note that   the connectedness of  $\pa \OM_n$  allows to impose a Dirichlet condition for  $\p_n$ on $\partial \OM_n$. This Dirichlet condition  can also be read from the formula  $\p_n=\frac{1}{2\pi}\ln |\Tc_n(x)|$, as $\Tc_n$ maps $\pd \OM_n$ to $\pd B(0,1)$. Therefore, by a classical energy estimate and Poincar\'e inequality applied in $B(0,R_0+1)$, we obtain that:
\[ \| \na(\chi\p_n) \|_{L^2(\OM_n)}^2 \leq \| \tilde \om_n \|_{L^2(\OM_n)}  \| \chi\p_n \|_{L^2(\OM_n\cap B(0,R_0+1))}\leq C_{R_0} \| \tilde \om_n \|_{L^2(\OM_n)}  \| \na(\chi\p_n) \|_{L^2(\OM_n)}.\]
Using that $\p_n$ and $\na \p_n$ converge uniformly to $\p$ and $\na \p$ in $B(0,R_0+1)\setminus B(0,R_0)$ (see Proposition \ref{propconvbiholo}), we get that $ \| \tilde \om_n \|_{L^2(\OM_n)}$ is uniformly bounded. This yields the existence of a constant $C$, depending only on $R_0$, such that
\begin{equation}\label{L2 harm}
\| \na \p_n \|_{L^2(\OM_n\cap B(0,R_0))} \leq C,\ \forall n.
\end{equation}

\subsection{Conclusion of the proof}\ 

The proof follows the one for bounded open set, taking care of integrability at infinity. We assume here that $p=\infty$, we fix $T>0$ and a compact $K$ in $\R^2$. We denote
\[ D : = K\cup B(0,\r_T),\]
with $\r_T$ defined in the previous subsection.

\bigskip

{\it a) Compactness of the rotational part.}

Extending $\psi^0_n$ by $0$ outside $\Omega$, we deduce from \eqref{na psi} and \eqref{psi loc} that
$$\| \na \psi^0_n \|_{L^2(\R^2)} \:  \le \: C_T, \quad   \| \psi^0_n(t,\cdot) \|_{H^1(D)} \:  \le \:  C_{T,K}  \quad \forall t\in [0,T],\  \forall n\geq N_K. $$

As it regards time derivatives, we observe that $\pa_t \psi^0_n$ satisfies 
\begin{equation*}
\Delta \left( \pa_t \psi^0_n \right) \: = \: \pa_t \omega_n \: =  \: - \div \left( u_n \omega_n\right) \: \mbox{ in } \: \Omega_n, \quad \pa_t \psi^0_n\vert_{\pa \Omega_n} = 0.  
\end{equation*}
Using the uniform $L^\infty([0,T], L^2(B(0,\r_T)))$ bound on $u_n$ (see \eqref{na psi} and \eqref{L2 harm}) and the $L^\infty$ bounds on $\omega_n$, we get
$$  \|  \pa_t \psi^0_n(t,\cdot) \|_{H^1(D)} \: \le \: C,  \quad \forall t\in [0,T],\  \forall n\geq N_K. $$ 
From these bounds and standard compactness lemma, there exists $\psi^0 \in W^{1,\infty}(0,T;  H^1(D))$, with $\na \psi^0 \in L^\infty(0,T; L^2(\R^2))$, 
 such that up to a subsequence: 
$$ \psi^0_{n} \rightarrow \psi^0 \: \mbox{ weakly* in } \: W^{1,\infty}(0,T; H^1(D)) \: \mbox{ and strongly in }  C^0(0,T; L^2(D)). $$
We now extend   $\omega_n$ by $0$ outside $\Omega_n$ for all $n$, so that  the sequence  $\left( \omega_{n} \right)_{n \in \N}$ is bounded in $L^\infty(\R^+; L^1\cap L^\infty(\R^2))$. Up to another extraction,  we deduce  that 
\begin{equation} \label{definitionomegaR2}
 \omega_n \rightarrow \omega \:  \mbox{ weakly *  in } \: L^\infty(\R^+; L^1\cap L^\infty(\R^2)). 
\end{equation}
From the weak convergence of $\psi^0_n$ and $\omega_n$, we infer that 
\begin{equation} \label{limitpsi0 ext}
 \Delta \psi^0(t, \cdot) = \omega(t,\cdot) \: \mbox{ in } \: {\cal D}'(\Omega \cap D), \: \mbox{ for almost every } \: t 
 \end{equation}
using again that any compact subset of $\Omega$ is included in $\Omega_n$ for $n$ large enough. 
 
 \medskip
Now, we use Proposition \ref{prop9}: as $(\Omega_n \cap B(0,\r_T))_{n \in \N}$ converges to $\Omega \cap B(0,\r_T)$ in the Hausdorff sense and  as the complement in a large closed ball $B$  of  $\Omega_n \cap B(0,\r_T)$ has at most $2$ connected components for all $n$, $(\Omega_n\cap B(0,\r_T))_{n \in \N}$ $\gamma$-converges to $\Omega\cap B(0,\r_T)$. Let $\chi$ be a cutoff function equal to $1$ in a neighborhood of the $O_n$'s, and to $0$ outside $B(0,\r_T)$.  By Proposition \ref{prop10},  $\chi \psi^0_n(t, \cdot)$ has for every $t$ a subsequence that converges weakly in $H^1_0(B(0,\r_T))$ to a limit in $H^1_0(\Omega\cap B(0,\r_T))$. Thus, for every $t\in [0,T]$,  $\chi\psi^0(t,\cdot)$ belongs to $H^1_0(\Omega \cap B(0,\r_T))$.  

\medskip
Finally, let us prove the strong convergence of $\psi^0_n$ to $\psi^0$ in $L^2(0,T; H^1(D))$ for all $T >0$. Note that for $R$ large enough, $\psi^0_n$ and $\psi^0$ are harmonic functions outside $D(0,R)$, with $\na \psi^0_n$ and $\na \psi^0$ both square integrable. By standard arguments, it implies 
$\psi^0_n, \psi^0 = cst + {\cal O}(1/|x|)$ and $|\na \psi^0_n|,  | \na \psi^0| = \mathcal{O}(1/|x|^2)$ at infinity. This makes the integrations by parts that follow rigorous. 

Going back to the equation  \eqref{psi0-ext}, we get
\begin{equation*} 
\int_0^T \int_{\R^2} |\na \psi^0_n|^2 =  \int_0^T \int_{\Omega_n}  |\na \psi^0_n|^2 = -\int_0^T \int_{\Omega_n} \omega_n \,   \psi^0_n = \: - \int_0^T \int_{D} \omega_n  \,  \psi^0_n \: \rightarrow - \int_0^T \int_D \omega \, \psi^0 
\end{equation*}
As we know from the previous paragraph  that $\chi \psi^0(t,\cdot)$ belongs to $H^1_0(\Omega\cap B(0,\r_T) )$ for every $t$, we can perform an energy estimate on \eqref{limitpsi0 ext} as well. We get 
$$ \int_0^T \int_{\R^2} |\na \psi^0|^2 =  \int_0^T \int_\Omega |\na \psi^0|^2 = - \int_0^T \int_\Omega \omega \, \psi^0 = - \int_0^T \int_D \omega \, \psi^0  $$
Hence, 
$$ \int_0^T \int_{\R^2} |\na \psi^0_n|^2  \: \rightarrow \:  \int_0^T \int_{\R^2} |\na \psi^0|^2  $$
which together with the weak convergence in $W^{1,\infty}(0,T; H^1(D))$ yields the strong convergence of $\psi^0_n$ to $\psi^0$ in $L^2(0,T; H^1(K))$. 

\begin{remark}
As in the bounded case, it is important that $\om^0$ belongs to $L^\infty$, i.e. $p=\infty$, to get uniform estimates on $\pa_t \psi^0_n$.
\end{remark}

\bigskip

{\it b) Compactness of the harmonic part.}

By the convergence results on $\{ \mathcal{T}_n \}$ (see Proposition \ref{propconvbiholo}), we obtain directly that $\psi_n=\frac{1}{2\pi}\ln |\Tc_n(x)|$ converges uniformly to $\psi$, resp. to $0$,  in any compact subset $K$ of $\Pi$ (the unbounded connected component of $\OM$),  resp.  of $\OM \setminus \Pi$ (the bounded connected components of $\OM$). As $\psi_n - \psi$ is a  harmonic function over $\Omega$, we infer from the mean-value theorem that the convergence in $L^\infty_{\loc}(\OM)$ implies the convergence in $H^1_{\loc}(\OM)$.

\bigskip

{\it c) Limit equation.}

We can now conclude the proof of Theorem \ref{theorem2}. Let $(u_n)_{n \in \N}$ be  the sequence of Euler solutions in $\Omega_n$, associated to the initial data $u_0^n$. Each field $u_n$ has the  Hodge decomposition  \eqref{hodge-ext}.  By diagonal extraction, it  {\em converges strongly in $L^2_{\loc}(\R^+ \times \overline{\OM})$  and weakly* in $L^\infty_{\loc}(\R^+; L^2_{\loc}( \overline{\OM}))$}  to the field
\begin{equation*}
u(t,x) \: = \: 
\left\lbrace\begin{aligned}
&\na^\bot \psi^0(t,x) \: + \: \ \alpha \na^\bot \psi(x), &\text{ if } x\in \Pi,  \\
&\na^\bot \psi^0(t,x), &\text{ if } x\in \OM \setminus \Pi.
\end{aligned}\right.
\end{equation*}
Note that  $\na^\perp \psi^0$ belongs to $L^\infty_{\loc}(\R^+; L^2(\Omega))$ (see \eqref{na psi})  whereas $\na^\perp \psi$ is only locally square integrable. It follows that  $u \in L^\infty_{\loc}(\R^+; L^2_{\loc}( \overline{\OM}))$.   From this explicit form, we deduce that $u$ is divergence free, tangent to the boundary, with a conserved circulation along the closed curve $J$. Moreover, inside $\Omega$, one has   
$$\curl u  \: = \:  \Delta \psi^0 \: =  \: \omega \in L^\infty(\R^+; L^1\cap L^\infty(\OM)).$$
The uniform estimate of the support of $\om_n$ means that $\om$ is also compactly supported. 

\medskip

Finally, let $\varphi \in  {\cal D}\left([0, +\infty), {\cal V}(\Omega)\right)$. For $n$ large enough, the support of $\varphi$ is included in $\Omega_n$ so that:
\begin{equation*} 
  \int_0^{\infty} \int_\Omega \left( u_n \cdot \pd_t \varphi +   (u_n \otimes u_n) : \na \varphi \right)  = -\int_\Omega u^0_n \cdot \varphi(0, \cdot) 
\end{equation*}
By the  strong $L^2_{\loc}$ convergence of $u_n$ to $u$, and also the strong $L^2_{\loc}$  convergence of $u^0_n$ to $u^0$, it follows that $u$ satisfies the weak form of the Euler equations \eqref{Eulerweak}. Note that, as in the case of bounded domains, the strong convergence of $u^0_n$ to $u^0$ relies on the uniqueness of the  Hodge decomposition for irregular open sets (outside one obstacle in this paragraph). For the sake of brevity, we leave to the reader to adapt the argument given in the previous section (see also  \cite[Proposition 2.1]{ift_lop_euler}).  

\medskip
{ Let us emphasize that this convergence in $L^2_{\loc}$ does not hold  for the situation studied in \cite{ift_lop_euler}  (one small obstacle shrinking to a point), or in  \cite{milton} (bounded domain with several holes,  one of them  shrinking to a point). In such  situations, the limit velocity is the sum of a smooth part and a harmonic part $x^{\perp}/|x|^2$, so it  does not even belong to $L^2_{\loc}$. }

\section{Initial vorticity in $L^p$}\label{section5}

We complete  in this section the proof of Theorems \ref{theorem1} and \ref{theorem2}, by treating the case of vorticities in $L^p$ for finite $p$ ($p > 2$ in the case of exterior domains). We do not need here to  introduce approximate domains $\Omega_n$: we regularize only the initial data, and   rely on the existence of weak solutions (in the original domain $\Omega$) for bounded vorticities,  as established in the previous sections.  
\bigskip

\subsection{Theorem \ref{theorem1} for $p>1$}\ 

Let $p > 1$, $u^0$ satisfying \eqref{initialdatabounded}.  Let $\omega^0 := \curl u^0$. We introduce a sequence of smooth functions $(\om_n^0)_{n\in \N}$  such that $\om_n^0 \to \om^0$ strongly in $L^p(\OM)$.  We remind that we have established  in Section \ref{section3} that  there is for each $n$ and each real $k$-uplet $c^1, \dots, c^k$ a unique $u^0_n \in L^2(\Omega)$ satisfying 
$$ \curl u^0_n = \om^0_n, \quad \gamma^i(u^0_n)  = c^i, \quad \forall i=1,\dots, k $$
together with  the divergence-free and tangency conditions. We choose here  $c^i  := \gamma^i(u^0)$.  We then denote by  $u_n$ a weak solution constructed in Section \ref{section3}. We denote $\omega_n := \curl u_n$.  

\medskip
Assuming as in Section \ref{section3}  that $u_n$ is  the limit of  a sequence of smooth solutions $u_{n,N}$ of Euler in smooth domains $\Omega_N$, we notice by \eqref{conservationlp} that:

\begin{equation} \label{borneomegalp}
 \| \om_n \|_{L^\infty(L^p(\OM))} \leq \liminf_{N\to \infty} \| \om_{n,N} \|_{L^\infty(L^p(\Omega_N))} \le  \| \om^0_n \|_{L^p(\OM)} \le C_p. 
 \end{equation}
Then we have, up to a subsequence, the weak $*$ convergence of $\om_n$ to some  $\om$ in $L^{\infty}(\R^+; L^p(\OM))$.

\medskip
Moreover, we proved that the velocity can be written as
\[  u_n(t,x) \: = \: \na^\bot \psi^0_n(t,x) \: + \: \sum_{i=1}^k \alpha^i_n(t) \na^\bot \psi^i(x) \]
where
\[ \psi^0_n \in L^\infty(\R^+;H^1_0(\OM)) \text{ and }  \Delta \psi^0_n(t, \cdot) = \omega_n(t,\cdot) \: \mbox{ in } \: {\cal D}'(\Omega), \: \mbox{ for almost every } \: t;\]
\[ \psi^i \: = \: \sum_{j=1}^k c^{i,j} \, \phi^i \text{ for all } i=1, \dots, k;\] 
\begin{equation}\label{alpha-bis}
 \alpha^i_n \: = \:   \int_{\Omega} \phi^i \, \omega_n \, dx \:  +  \gamma^i,
\end{equation}
with
\[ \phi^i \in  H^1_0(\tilde \OM), \  \Delta \phi^i = 0 \quad \mbox{ in } \: \Omega, \quad \phi^i\vert_{\pa \mathcal{C}^j} = \d_{ij}   \: \mbox{ in a weak sense, see Section \ref{section3}}, \]
and
\[ C = \left( c^{i,j} \right)_{1 \le i,j \le k} = -\left( \int_{\Omega} \na \phi^i \cdot \na \phi^j \right)_{1 \le i,j \le k}^{-1}. \]
 Note that $\phi^i$, $\psi^i$ and $C$ do not depend on $n$.

 \bigskip
 By the energy estimate, we obtain that $\| \na \psi^0_n(t,\cdot) \|_{L^2(\Omega)}^2  \leq   \| \omega_n(t, \cdot) \|_{H^{-1}(\Omega)} \| \psi^0_n(t, \cdot) \|_{H^1(\Omega)}$, which implies by \eqref{borneomegalp} and  the Poincar\'e inequality on a big ball $D$ that $\psi^0_n$ is uniformly  bounded in $L^{\infty}(\R^+; H_0^1(\OM))$. Also by \eqref{borneomegalp}, the sequences $(\alpha^i_n)_{n \in \N}$  are bounded in $L^\infty(\R^+)$.  

\medskip
Therefore, we can write $u_n=\na^\perp \p_n$ with $\p_n$ bounded in $L^\infty(\R^+;H^1(\OM))$ and $u_{n}^0$ bounded in $L^2(\OM)$. By these estimates, we extract a subsequence such that $u_{n}^0\to u^{0}$ weakly in $L^2(\OM)$, $u_n\to u$ weakly* in $L^\infty(\R^+;L^2(\OM))$, which implies that $u$ verifies the divergence-free and tangency conditions \eqref{imperm2}.

\medskip
The last step consists in proving that $u$ verifies the momentum equation \eqref{Eulerweak} for any divergence free test function $\varphi \in {\cal D}\left([0, +\infty) \times \Omega\right)$. Let us fix such a test function, $T$ big enough such that $\varphi(t,\cdot)\equiv 0$ for $t\geq T$. We set $\Omega'$ a smooth set such that $\supp \varphi(t,\cdot)\subset \Omega'\Subset \Omega$ for any $t\in [0,T]$, and we are looking for a subsequence such that we can pass to the limit in the momentum equation. We denote by $\mathbb{P}_{\Omega'}$ the Leray projector onto $\mathcal{H}(\Omega')$ (see \eqref{impermbis} for the definition) and we set
\[ {\cal V}(\Omega'):=\text{completion of } \{ \F \in{\cal D}(\Omega')\ | \ \div \F =0 \} \text{ in the norm of } H^1.
\]
By the standard properties of the Leray projection (see e.g. \cite{Galdi}), we decompose $u_n$ in $\Omega'$ as:
\[
u_n = \mathbb{P}_{\Omega'} u_{n} +\nabla q_{n}.
\]
As the projector is orthogonal in $L^2(\Omega')$, we have that $\mathbb{P}_{\Omega'} u_{n}$ and $\nabla q_{n}$ are uniformly bounded in $L^\infty(\R^+;L^2(\OM'))$, then weakly converge, up to a subsequence, to $\mathbb{P}_{\Omega'} u$ and $\nabla q$, respectively, with $u=\mathbb{P}_{\Omega'} u+\nabla q$. As $\Omega'$ is smooth and $\omega_{n}\in L^\infty(\R^+;L^p(\OM'))$, it comes from Calderon-Zygmund inequality on $\OM'$ that $\mathbb{P}_{\Omega'} u_{n}$ is bounded in $L^\infty(\R^+;W^{1,p}(\Omega'))$. Using the equation verified by $u_{n}$, we compute for any divergence free function $\F \in {\cal D}((0, T) \times \Omega')$:
\begin{eqnarray*}
\langle \pa_{t}\mathbb{P}_{\Omega'} u_{n},\F  \rangle &=& -\int_{0}^T\int_{\Omega'} \mathbb{P}_{\Omega'} u_{n}(t,x) \cdot \partial_{t}\F \, dxdt = -\int_{0}^T\int_{\Omega'}  u_{n}(t,x) \cdot \partial_{t}\F \, dxdt \\
&\leq& \| u_{n}\|_{L^\infty(\R^+;L^2(\OM'))}^2 \| \nabla \F \|_{L^2((0,T)\times \OM')} \sqrt{T}
\end{eqnarray*}
which implies that $\pa_{t}\mathbb{P}_{\Omega'} u_{n}$ is bounded in $L^2((0,T); \cal V'(\Omega'))$. Finally, by the Aubin-Lions lemma, we get the strong compactness of $\mathbb{P}_{\Omega'} u_{n}$ in $L^2((0,T)\times \Omega')$, where we have used that $\cal H(\Omega')$ embeds continuously in $\cal V'(\Omega')$.

Before to pass to the limit, we note the following equality if $\Delta p^h=0$:
\begin{equation}\label{algebra}
\int_{\OM'} \nabla p^h\otimes\nabla p^h : \nabla \varphi  =   -\int_{\Omega'}\left( \frac{1}{2} \na |\na p^h|^2 \cdot \varphi + \Delta p^h \na p^h \cdot \varphi \right)=   0,
\end{equation}
because $\varphi$ is divergence free and compactly supported in $\Omega'$. Such a relation can be applied with $p^h=q_{n}$:
\[
\int_0^{\infty} \int_\Omega \left( u_{n} \cdot \pd_t \varphi +   (\mathbb{P}_{\Omega'} u_{n} \otimes u_{n} + \nabla q_{n} \otimes \mathbb{P}_{\Omega'}  u_{n}) : \na \varphi \right)  = -\int_\Omega u^0_{n} \cdot \varphi(0, \cdot).
\]
which leads to 
\[\int_0^{\infty} \int_\Omega \left( u \cdot \pd_t \varphi +   (\mathbb{P}_{\Omega'}  u \otimes u+ \nabla q \otimes \mathbb{P}_{\Omega'}  u) : \na \varphi \right)  = -\int_\Omega u^0 \cdot \varphi(0, \cdot).
\]
Using \eqref{algebra} with $p^h=q$, we conclude that $u$ verifies the momentum equation \eqref{Eulerweak}, which ends the proof of Theorem \ref{theorem1} for $p\in (1,\infty)$.

\begin{remark}
Actually, the solution constructed with $\curl u^0 \in L^p(\OM)$, $p\in (1,\infty]$, has the following additional properties:
\begin{itemize}
\item the weak circulations are conserved;
\item 
in the special case $p=\infty$,  the momentum equation \eqref{Eulerweak} is verified for all test functions whose support intersects the boundary, i.e. \eqref{Eulerweak} holds for any $\varphi \in {\cal D}\left([0, +\infty) \times \overline{\Omega}\right)$.
\end{itemize}
\end{remark}

\subsection{Theorem \ref{theorem2} for $p>2$}\ 

To go  from $p=\infty$ to $p > 2$, one  follows  the lines of  the bounded case. Let us remark that  for solutions in Theorem \ref{theorem2}, we have that $\a_n(t) = \a_n = \int_{Ext(J)} \om^0_n + \int_J u^0 \cdot \tau$ which tends easily to $\a  :=  \int_{Ext(J)} \om^0 + \int_J u^0 \cdot \tau$.

\medskip

Next we prove the convergence of the rotational part. As Lemma \ref{est kernel} holds for $p\in (2,\infty]$ (the restriction $p > 2$ comes from this lemma), then we control uniformly the size of the support of $\om_n$, and we get a uniform estimate of $\p_n^0$ in $L^\infty((0,T);H^1(D))$, with $D:=K\cup B(0,\r_T)$ (see Subsection \ref{sub support}). Hence, we deduce that $u_n$ converges weak-$*$ to $u$ in $L^\infty(\R^+;L^2(D))$.

\medskip

The last step is the strong convergence of $u_n$, which can be proved exactly with the same arguments as above: on each ${\cal O}\Subset \OM$, decomposing $u_n$ as $\mathbb{P}_{\cal O}u_n+\nabla q_n$.

\section{Final remarks}\label{section6}

\subsection{Domain continuity for Euler}\ 

{ Theorems \ref{theorem1} and \ref{theorem2} yield existence of global weak solutions in singular open set. However, their proof  yields more, namely some domain continuity for the Euler equations. It shows that solutions of Euler in  
$$\Omega_n \: :=   \:  \widetilde \Omega_n \: \setminus \: \left( \cup_{i=1}^k \overline{O_n^i}\right), \quad \mbox{ resp. } \:   \Omega_n \: :=   \:  \R^2 \setminus \overline{O_n}$$ 
converge  to solutions of Euler in 
$$ \Omega\: := \: \widetilde \Omega \: \setminus \: \left( \cup_{i=1}^k {\cal C}^i\right), \quad \mbox{ resp. } \:   \Omega \: :=   \:  \R^2 \setminus {\cal C}.  $$
We discuss here some consequences  of this convergence result.}

\bigskip
{\it Rugosity.} A typical  problem in rugosity theory is the following: let $\OM$ be a smooth domain with a flat wall $y=0$. Let   $\OM_\e$ be obtained from $\OM$ by a boundary perturbation of the form $y=\e^\a \cos(x/\e)$ ($\a>0$ fixed). What is the asymptotic behaviour of the flow  in $\OM_\e$ as $\e\to 0$ ? In the case of viscous flows,  it has been shown that  there is  a drastic  effect of the rugosity at the limit, see \cite{Simon, Bucur}.
In the opposite direction, {\em one can  deduce} from our analysis that such effect does not hold for ideal incompressible flows: the solution $u_\eps$ of the Euler equations on $\OM_\eps$ converges to the solution $u$ of the Euler equations on $\OM$.

\bigskip
{\it Trapping of a flow.} The complements of the domains $\Omega_n$ and $\Omega$ that we consider have the same number of connected components. Thus, the domain continuity that we show does not extend  to  the fusion of two obstacles as in Figure 1. In such a case, we do not pretend that $u_n$ solution in $\OM_n$ (see Figure 1) converges to $u$ solution in $\OM$. Actually, we guess that it does not hold because of the Kelvin's circulation theorem.

\begin{figure}[ht]
\hspace{1,5cm} \includegraphics[height=4.5cm]{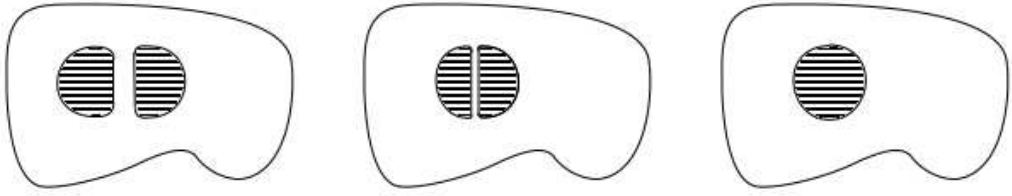}
\caption{fusion of two obstacles.}
 \end{figure}

\medskip
However, an example that we can include in our analysis is an obstacle $\overline{O_n}$ which closes on itself (see Figure 2). In this picture, although $\OM_n$ as a unique connected component,  $\OM$ has several connected components. Here, $\mathcal{C}$ is a Jordan curve, and it is an example of a compact set obtain as a  Hausdorff limit, but not as a decreasing sequence of smooth simply connected domains. In such a setting, the present  work still shows  that $u_n$ solution in $\OM_n$ (see Figure 2) converges to $u$ solution in $\OM$.

\begin{figure}[ht]
\hspace{1,5cm}\includegraphics[height=4.5cm]{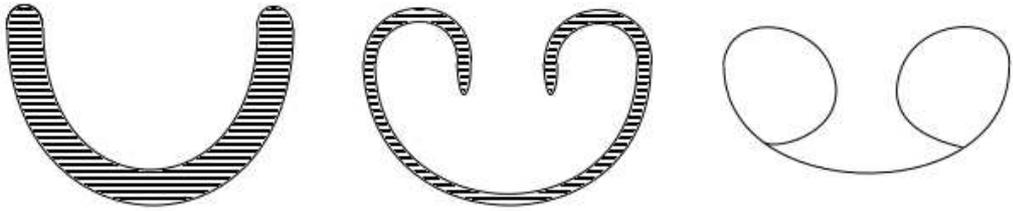}
\caption{$\overline{O_n}$ tends to $\mathcal{C}$ in the Hausdorff sense.}
 \end{figure}

\subsection{The case of the Jordan arc}\

In this subsection, we pay special attention to the  case of a smooth  Jordan arc ${\cal C}$. We shall  denote $0_1$ and  $0_2$ the endpoints of the arc.  As mentioned earlier, this geometry has already been investigated by the second author in \cite{lac_euler}.  In that article, the existence of Yudovich- type solutions  is established through an approximation  by regular domains $\Omega_\eps$. The corresponding regular solutions $u_\eps$ and their curl $\omega_\eps$ are then truncated smoothly over a size  $\eps$ around the obstacle. The resulting truncations 
$\tilde u_\eps$ and $\tilde \omega_\eps$, defined over the whole of $\R^2$,  are shown to converge in appropriate topologies to   the solutions $\tilde u$, $\tilde \omega$  of the system
\begin{equation} \label{vorticityformulation2}
\left\{
\begin{aligned}
& \div \tilde u = 0,    \quad t > 0, \:  x \in \R^2, \\
& \pa_t \tilde \omega + \tilde u \cdot \na  \tilde \omega = 0, \quad t > 0, \:  x \in \R^2, \\
& \tilde \omega \: := \:  \curl \tilde u \:  - \:  g_{\tilde \omega} \delta_{\cal C}, \quad t > 0, \:  x \in \R^2, 
\end{aligned}
\right.
\end{equation}
(plus a circulation condition). This is an Euler like system, modified by a Dirac mass along the arc. The density function $g_{\tilde \omega}$ is given explicitly in terms of $\tilde \omega$ and ${\cal C}$. Moreover, it is shown that it is equal to the jump of the tangential component of the velocity across the arc. We refer to \cite{lac_euler} for all necessary details. Our point in this section is to link this formulation in the whole space to the classical formulation of the Euler equations in $\Omega$, see  \eqref{Eulerweak}-\eqref{imperm}.

\medskip
More precisely,  let  $u$ be the solution of  \eqref{Eulerweak}-\eqref{imperm}  built in Section \ref{section3}, and $\omega := \curl u$. We want to show that extending  $u$ and  $\omega$ by $0$ yields a solution of \eqref{vorticityformulation2} in $\R^2$. 
Therefore, we first notice that these extensions (still defined by $u$ and $\omega$) satisfy 
$$ u \in L^\infty_{\loc}(\R^+; L^2_{\loc}(\R^2)), \quad \omega \in L^\infty(\R^+; L^\infty(\R^2) \cap L^1(\R^2)). $$
It follows easily from  the estimates \eqref{na psi}-\eqref{L2 harm}, and  \eqref{conservationlp}-\eqref{convergenceomegan}.
Then, we remark that  $u= \na^\bot \psi^0 + \alpha \na^\perp \psi$ is  clearly divergence free over the whole of $\R^2$.  

\medskip
We now turn to the transport equation for the vorticity. Taking  $\varphi = \na^\bot \psi$ in \eqref{Eulerweak}, with some $\psi$ compactly supported in $(0,+\infty) \times \Omega$, we first obtain 
\begin{equation} \label{vorticity1}
 \pa_t \omega +   u \cdot \na \omega  = 0, \quad \mbox{ in } \:   (0,+\infty) \times \Omega
 \end{equation}
in the distributional sense.   Let now 
$$\varphi \in {\cal D}\left([0,+\infty) \times (\R^2\setminus\{0_1, 0_2\})\right)$$
be a  scalar test function. We want to prove that 
\begin{equation} \label{vorticity2} 
\int_{\R^+} \int_{\R^2} \pa_t \varphi \,  \omega +  \int_{\R^+} \int_{\R^2}  \na \varphi \cdot (\omega u ) =  \int_{\R^2} \varphi(0,\cdot) \,  \omega^0,
\end{equation}
meaning that the transport equation is satisfied over $\R^2\setminus\{0_1,0_2\}$.  We introduce a curvilinear coordinate $s \in (0,S)$  and a transverse coordinate $r \in  [-R,R] $, so that in a neighborhood of ${\cal C}\setminus\{0_1, 0_2\}$, one has  $x = J(s) \: + \: r \, \nu(s)$, $\nu$ a normal vector field. In view of \eqref{vorticity1}, we can assume with no loss of generality that $\varphi$ is compactly supported in this neighborhood. We then  consider a truncation function  that reads
$$ \varphi_\eps(t,x) = \varphi(t,x) (1-\chi(r/\eps))  $$ 
where $\chi \in C^\infty_c(\R)$,   $\chi \equiv 1$ near $0$. 
One can use $\varphi_\eps$ as a test function in \eqref{vorticity1}. Setting
\[\chi_\eps(x) := \chi(r/\eps),\]
to prove \eqref{vorticity2}, it remains to prove that 
$$ \int_{\R^+} \int_{\R^2} \pa_t \varphi \, \chi_\eps \,  \omega +  \int_{\R^+} \int_{\R^2}  \na (\varphi \, \chi_\eps) \cdot (\omega u )- \int_{\R^2} \varphi(0,\cdot)\, \chi_\eps \,  \omega^0 \rightarrow 0, \quad \mbox{ as } \: \eps \rightarrow 0 . $$
The only difficult term is 
$$ I_\eps \: :=  \: \int_{\R^+} \int_{\R^2}  (\varphi \,  \na \chi_\eps) \cdot (\omega u ).  $$
We remind that the streamfunction $\eta = \psi^0 + \alpha \psi$  associated to  $u$ satisfies $\Delta \eta = \omega$ in $\Omega$, and that it is constant at ${\cal C}$ by the impermeability condition. As $\omega$ is bounded, it follows from elliptic regularity that $\eta$ has $W^{2,p}$ regularity for all finite $p$ on each side of ${\cal C}$, {\em away from the endpoints $0_1, 0_2$}. In particular, one has 
\begin{equation} \label{uw1p}
\| u \|_{W^{1,p}(K_\eps)} \: \le \:  C_p 
\end{equation}
over the support $K_\eps$ of $\varphi  \na \chi_\eps$. Denoting $u_\nu(x) := u(x) \cdot \nu (s)$ the ``normal'' component of $u$, one has 
\begin{align*}
| I_\eps  | \: & \le \:  C \, \int_{K_\eps} \frac{1}{\eps} |\chi'(r/\eps)| \, |u_\nu (x)| \, dx \: \le \: C \,  \int_{K_\eps} \frac{r}{\eps} |\chi'(r/\eps)| \, \frac{|u_\nu (x)|}{r}  \, dx  \\
& \le \: C \sup_\theta \left(\theta \, |\chi'(\theta)|\right)  \,   \int_{K_\eps} \frac{|u_\nu (x)|}{r} \, dx\: \le \: C' \sqrt{ \int_{K_\eps} \, dx } \,  \sqrt{ \int_{K_\eps}  | \na u_\nu (x)|^2 \, dx}  
\end{align*}
where the last bound comes from the  Hardy inequality, applied on each side of ${\cal C}$ to $u_\nu $ (which vanishes at ${\cal C}$ by the impermeability condition).   It follows from \eqref{uw1p} that $I_\eps$ vanishes to zero with $\eps$, as expected.  

\medskip
Thus, to establish the transport equation for the vorticity on the whole plane, it remains to handle the neighborhood of the endpoints $0_1, 0_2$,  say $0_1$.  This time, we introduce the truncation
$$\chi_\eps(x) := \chi\left(\frac{x-0_1}{\eps}\right) \mbox{ \:  with $\chi \in C^\infty_c(\R^2)$, $\: \chi=1$  near $0$.}$$ 
As before, one is left with showing that 
$$ I_\eps \: :=  \: \int_{\R^+} \int_{\R^2}  (\varphi \,  \na \chi_\eps) \cdot (\omega u )   $$
goes to zero with $\eps$. But this time, as $\na \chi_\eps$ is uniformly bounded in $L^2$, one has  the  simple inequality 
$$ | I_\eps | \: \le \: C \,  \| \na \chi_\eps \| \, \| \omega u \|_{L^2(K_\eps)} \: \le \: C'   \| \omega u \|_{L^2(K_\eps)}   $$
where $K_\eps$ is the support of $\chi_\eps$. The r.h.s. goes to zero by  Lebesgue dominated convergence theorem, and yields the result. 

\bigskip
 {\em Eventually, we have  to establish the third line of \eqref{vorticityformulation2}}, which expresses $\omega$ in terms of 
$u$ and a Dirac mass along the arc. Again, we notice  that the streamfunction $\eta$ has $W^{2,p}$ regularity for all finite $p$ on each side of the arc, away from its endpoints. This implies that the velocity $u$ has a trace from each side of the arc, denoted by  $u_{\pm}$.  These traces belong to $W^{1-1/p,p}_{\loc}(int({\cal C}))$ for any finite $p$. By the impermeability condition, only the tangential component of these traces is non-zero. 
Let now  $\varphi \in  C^\infty_c(\R^2\setminus\{0_1,0_2\})$ a scalar test function. Testing this function with the relation $\omega = \curl u$ (that clearly holds in the strong sense in $\R^2\setminus{\cal C}$), and integrating by parts on each side of the arc, we end up with 
$$ \int_{\R^2} \omega \, \varphi \: = \: -\int_{\R^2} u \cdot \na^\bot \phi  \: + \: \int_{{\cal C}} [ u_\tau ] \varphi,  $$
almost surely in $t$,  where  {\em $[ u_\tau ]$ denotes the jump of the tangential component}: if $n$ is the normal going from side $+$ to side $-$, $[ u_\tau ] \: := \: (u_+ - u_-) \cdot \nu  ^\bot$. The last equation can be written 
$$ \omega \: = \: \curl u \: - \: g_\omega \delta_{\cal C}   \:  \mbox{ in } \R^2\setminus\{0_1,0_2\}   $$
in the sense of distributions, where  $g_\omega(s) := [ u_\tau ](s)$ ($s$ the curvilinear coordinate). 

\medskip
The last step is to go from $\R^2\setminus\{0_1,0_2\}$ to $\R^2$.  Therefore, we introduce again truncation functions near the endpoints: say
$$\chi_\eps(x) := \chi\left(\frac{x-0_1}{\eps}\right) \mbox{ \:  with $\chi \in C^\infty_c(\R^2)$, $\: \chi=1$  near $0$.}$$ 
As before, the point is to show that 
$$\int_{\R^2} \omega \, \varphi \chi_\eps, \quad \int_\R^2 u \cdot \na^\bot \phi \chi_\eps, \quad \mbox{ and } \int_{{\cal C}} [ u_\tau ] \varphi \chi_\eps $$
all go to zero with $\eps$. The only annoying quantity is the third one: it requires  a control on the jump function $[ u_\tau ]$ {\em up to the endpoint $0_1$}. Therefore, we use  results related to elliptic equations in polygons, see \cite{Kondra,Mazya}. Indeed, up to a smooth change of variable, the Laplace equation for $\eta$ in  $\R^2\setminus{\cal C}$ turns into a divergence form elliptic equation in the exterior of a slit.  In particular, it follows from the results in \cite{Kondra,Mazya} that $u = \na^\bot \eta$ decomposes into $u_1 + u_2$, where $u_1$ behaves like $1/|x - 0_i|^{1/2}$ near the endpoint $0_i$, and $u_2$ has $W^{1,p}_{\loc}(\R^2\setminus{\cal C})$ regularity  for all $p < 2$. This allows to conclude that  $\int_{{\cal C}} [ u_\tau ] \varphi \chi_\eps$ goes to zero with $\eps$. This concludes the proof. 

\subsection{Extension to Delort's solutions} \ 

Looking closer at the proof of Theorem \ref{theorem1} for general $p > 1$ (see Section \ref{section5}), we see that uniform bounds on the field $u_n$ in $L^\infty L^2$ only require uniform bounds on $\curl u^0_n$ in $H^{-1} \cap L^1$. From there,  one can recover the appropriate initial data, tangency condition and divergence-free condition. Moreover,  the obtention of the Euler equation \eqref{Eulerweak} on the limit $u$ relies on local properties away from the boundary. Hence, one can replace  our compactness (Aubin-Lions) arguments by the analysis led by Delort in \cite[section 2.3, p582]{delort}. Consequently, it is possible to obtain an analogue of Delort's theorem (solutions with initial vorticity in $H^{-1}_{comp}(\R^2) \cap \mathcal{M}(\R^2)$ of definite sign)   in our singular domains.

\bigskip

\noindent
 {\bf Acknowledgements.} The first author is partially supported by the Agence Nationale de la Recherche, Project RUGO,  grant ANR-08-JCJC0104. The second author is partially supported by the Agence Nationale de la Recherche, Project MathOc\'ean, grant ANR-08-BLAN-0301-01. The authors are partially supported by the Project ``Instabilities in Hydrodynamics'' financed by Paris city hall (program ``Emergences'') and the Fondation Sciences Math\'ematiques de Paris.{ The authors are also grateful to Michel Pierre for hints on Proposition \ref{approxlemma}.}

\appendix 
\section{Sobolev capacity}  \label{app_capacity}
We recall here basic notions on Sobolev capacity, taken from \cite{henrot}. Let $E \subset \R^N$, $\: N \ge 1$. The capacity of $E$ (with respect to the Sobolev space $H^1(\R^N)$) is defined by 
$$ \capa(E) \: := \: \inf \{ \| v \|^2_{H^1(\R^N)}, \: v \ge 1 \: \mbox{ a.e.    in a neighborhood of } E\},  
$$
with the convention that $\capa(E)= +\infty$ when the set at the r.h.s. is empty. 
 The capacity is not a measure, but has similar good properties: 
 \begin{proposition} \label{prop5} \
 \begin{enumerate}
 \item $A \subset B \: \Rightarrow \: \capa(A) \le \capa(B). $
 \item Let $(K_n)_{n \in \N}$ a decreasing sequence of compact sets, with $K = \cap K_n$. Then, \\
 $ \capa(K)  = \lim \capa(K_n). $
 \item Let $(E_n)_{n \in \N}$ an increasing sequence of sets, with $E = \cup E_n$. Then, \\ 
 $\capa(E)  = \lim \capa(E_n). $
 \item (Strong subadditivity) For all sets $A$ and $B$, one has 
 $$ \capa(A \cup B) + \capa(A \cap B) \: \le \: \capa(A) + \capa(B). $$
 \end{enumerate}
 \end{proposition}
 
 \medskip
 \noindent If $D$ is a bounded open set of $\R^N$, one can also  define a capacity relatively to $D$: for 
$E \subset D$, 
$$ \capa_D(E) \: := \: \inf \left\{ \| \na v \|^2_{L^2(D)}, \: v \in  H^1_0(D),  \: v \ge 1 \mbox{ a.e.    in a neighborhood of } E \right\},  $$
with the same convention as before.  It is clear from this definition and the Poincar\'e inequality that 
 $ \capa(E) \: \le \: C \, \capa_D(E)$.

\medskip
For nice sets  $E$ in $\R^N$, the capacity of $E$ can be thought very roughly as some $n-1$ dimensional Hausdorff measure of its boundary. More precisely:
\begin{proposition} \label{prop6}\ 
\begin{enumerate}
\item For all compact  set $K$ included in a bounded open set $D$, \\ 
$\: \capa(K) = \capa(\pa K)$.
\item If $E \subset \R^N$ is contained in a manifold of dimension $N-2$, then $\capa(E) = 0$.
\item  If $E \subset \R^N$ contains a piece of some smooth hypersurface (manifold of dimension N-1), then $\capa(E) > 0$. 
\end{enumerate}
\end{proposition}

The last result concerns $H^1_0(\OM)$. When $\OM$ is a smooth open set,  $H^1_0(\OM)$ can be defined as the set of function in $H^1(\R^2)$ which are equal to zero almost everywhere in $\R^2\setminus \OM$. But this result does not hold  for general open sets $\OM$.  To generalize such a characterization, the notion of capacity is appropriate.
\begin{proposition}\label{prop11}
Let $D$ and $\OM$ be open sets such that $\OM \subset D$. Then
\[ \Bigl( v\in H^1_0(\OM)\Bigl) \iff \Bigl( v\in H^1_0(D) \text{ and } v = 0 \text{ quasi everywhere in } D\setminus \OM \Bigl),\]
which means that $v=0$ except on a set with zero capacity.
\end{proposition}

\section{Hausdorff convergence} \label{app_hausdorff}
We recall here basic notions of Hausdorff topology, taken from \cite{henrot}. We first introduce the Hausdorff distance for compact sets. Let ${\cal K}$ the set of all non-empty compact sets of $\R^N$, $N \ge 1$. For $K_1, K_2 \in {\cal K}$, we define
$$ d_H(K_1,K_2) \: := \: \max\left(\rho(K_1,K_2), \rho(K_2,K_1)\right), \quad      \rho(K,K') \: := \: \sup_{x \in K} d(x,K'). $$
It is an easy exercise to show that $d_H$ defines a distance on ${\cal K}$. Sequences that converge with respect to this distance are said to converge in the Hausdorff sense. One has the following basic properties
\begin{proposition} \label{prop1}\ 

\begin{enumerate}
\item A decreasing sequence of non-empty compact sets converges in the Hausdorff sense to its intersection. 
\item An increasing sequence of non-empty compact sets converges in the Hausdorff sense to the closure of its union.
\item Inclusion is stable for convergence in the Hausdorff sense.  
\item The Hausdorff convergence preserves connectedness. More generally, if $(K_n)_{n \in \N}$ converges to $K$, and $K_n$ has at most $p$ connected components, $K$ has at most $p$ connected components. 
 \end{enumerate}
\end{proposition}
A remarkable feature of the Hausdorff topology is given by the following 
\begin{proposition} \label{prop2}
Any bounded sequence of $({\cal K}, d_H)$ has a convergent subsequence. 
\end{proposition}

\medskip
From the  Hausdorff topology on ${\cal K}$, one can define a Hausdorff topology on {\em confined} open  sets, that is on all open sets included in some big given compact. Thus, let $B$ some compact domain in $\R^N$, $N \ge 1$, and ${\cal O}_B$ the set of all open sets included in $B$. The Hausdorff distance on  ${\cal O}_B$ is defined by: 
$$ d_H(\Omega_1, \Omega_2) \: := \: d_H(B\setminus\Omega_1, B\setminus\Omega_2) $$
 the r.h.s refering to the Hausdorff distance for compact sets. Let us note that this distance does not really depend on $B$: that is, for $B \subset B'$ two compact sets, and  $\Omega_1, \, \Omega_2$ in  ${\cal O}_B$,
 $$ d_H(B'\setminus\Omega_1, B'\setminus\Omega_2) \: = \: d_H(B\setminus\Omega_1, B\setminus\Omega_2). $$
\begin{proposition} \label{prop3}\ 
\begin{enumerate}
\item An increasing sequence of (confined) open sets converges in the Hausdorff sense to its union. 
\item A decreasing  sequence of (confined) open sets converges in the Hausdorff sense  to the interior of its intersection.
\item Inclusion is stable for convergence in the Hausdorff sense. 
\item Finite intersection is stable  for convergence in the Hausdorff sense
\item Let $(\Omega_n)$ a sequence that converges to $\Omega$ in the Hausdorff sense. Let $x \in \pa \Omega$. There exists a sequence $(x_n)$ with $x_n \in\pa \Omega_n$ that converges to $x$. 
\item Let $(\Omega_n)_{n \in \N}$ a sequence in ${\cal O}_B$. There exists an open set $\Omega \in {\cal O}_B$ and a subsequence $(\Omega_{n_k})_{k \in \N}$ that converges to $\Omega$ in the Hausdorff sense
 \end{enumerate}
\end{proposition}
Let us note that the Hausdorff convergence of open sets, contrary to the one of compact sets, does not preserve connectedness. Let us finally point out the following result, to be used later on:
\begin{proposition}  \label{prop4}
If  $(\Omega_n)_{n \in \N}$  converges in the Hausdorff sense to $\Omega$ and $K$ is a compact subset of $\Omega$, then there exists $n_0$ such that $\Omega_n \supset K$ for $n \ge n_0$. 
\end{proposition}

\section{$\gamma$-convergence of open sets}  \label{app_gammaconv}
Let $D$ be a bounded open set.  Let  $(\Omega_n)_{n \in \N}$ be  a sequence of open sets  included in  $D$. One says that  $(\Omega_n)_{n \in \N}$ $\gamma$-converges to $\Omega \subset D$ if for any $f \in H^{-1}(D)$, the sequence of  solutions $\p_n \in H^1_0(\Omega_n)$ of
$$ -\Delta \p_n = f \: \mbox{ in } \:  \Omega_n, \quad \p_n\vert_{\pa \Omega_n} = 0.$$
converges in $H^1_0(D)$  to the solution $\p \in H^1_0(\Omega)$ of
$$ -\Delta \p = f \: \mbox{ in } \:  \Omega, \quad \p\vert_{\pa \Omega} = 0.$$

In this definition, $H^1_0(\Omega)$ and $H^1_0(\Omega_n)$ are seen as subsets of $H^1_0(D)$, through extension by zero. In a dual way, $H^{-1}(D)$ is seen as a subset of 
$H^{-1}(\Omega_n)$ and $H^{-1}(\Omega)$. As for the Hausdorff convergence of open sets, the definition of $\gamma$-convergence does not depend on the choice of the confining set $D$.

\medskip
The notion of $\gamma$-convergence is extensively discussed in \cite{henrot}. The basic example of $\gamma$-convergence is given by  increasing sequences: 
\begin{proposition}  \label{prop8}
If $(\Omega_n)_{n \in \N}$ is an increasing sequence in $D$, it $\gamma$-converges to $\Omega \: = \:  \cup \, \Omega_n$. More generally, if  $(\Omega_n)_{n \in \N}$ is included in $\Omega$ and converges to $\Omega$ in the Hausdorff sense, then it $\gamma$-converges to $\Omega$. 
\end{proposition}
In general, Hausdorff converging sequences are not $\gamma$-converging. We refer to \cite{henrot} for counterexamples, with domains $\Omega_n$ that have  more and more holes as $n$ goes to infinity. This kind of  counterexamples, reminiscent of homogenization problems, is the only one in dimension 2, as proved by Sverak \cite{sverak}:
\begin{proposition} \label{prop9}
Let  $(\Omega_n)_{n \in \N}$ be a sequence of open sets in $\R^2$, included in $D$.  Assume that the number of connected components of  $D \setminus \Omega_n$ is bounded uniformly in $n$. If  $(\Omega_n)_{n \in \N}$ converges in the Hausdorff sense to $\Omega$, it $\gamma$-converges to $\Omega$.  
\end{proposition}  
This result will be a crucial ingredient  of the next sections.

\medskip
One can characterize the $\gamma$-convergence in terms of the Mosco-convergence of $H^1_0(\Omega_n)$ to $H^1_0(\Omega)$. Namely: 
\begin{proposition} \label{prop10}
 $(\Omega_n)_{n \in \N}$ $\gamma$-converges to $\Omega$ if and only if the following two conditions are satisfied: 
 \begin{enumerate}
 \item For all $\p \in H^1_0(\Omega)$, there exists a sequence $(\p_n)_{n \in \N}$  with $\p_n$ in $H^1_0(\Omega_n)$ that converges to $\p$. 
 \item For any sequence $(\p_n)_{n \in \N}$ with $\p_n$ in $H^1_0(\Omega_n)$, weakly converging to $\p$ in $H^1_0(D)$, $\p \in H^1_0(\Omega)$. 
 \end{enumerate}
 \end{proposition}
One can also  characterize  $\gamma$-convergence with capacity, see \cite[Proposition 3.5.5 page 114]{henrot}. Let us finally  mention that the notion of $\gamma$-convergence of open sets is related to the more standard $\Gamma$-convergence of Di Giorgi. Losely, $\Omega_n$ $\gamma$-converges to $\Omega$ if the corresponding Dirichlet energy functional $J_{\Omega_n}$ $\Gamma$-converges to $J_\Omega$: see \cite[section 7.1.1]{henrot} for all details.

\section{Convergence of biholomorphisms in exterior domains} \label{app_biholo}

We remind here the notion of kernel convergence introduced by Caratheodory in 1912, see \cite[p28]{pomm-1} (the word {\em domain} refers to a connected open set): 
\begin{definition}
Let $(F_n)$ be a sequence of  domains,  with $0 \in F_n$ for all $n$. Its kernel $F$ (with respect to $0$) is  the set consisting of $0$ together with all points $w \in \C$ that satisfy: {\em there exists a domain $H$ including $0$ and $w$ such that $H \subset F_n$ for all  $n$ large enough.}

\medskip   
If $F$ is the kernel of any subsequence of $(F_n)$, we say that $(F_n)$ converges to $F$ in the kernel sense. 
\end{definition}

This  type of geometric convergence is related to the famous Caratheodory theorem, see \cite[Theorem 1.8, p29]{pomm-1}:
\begin{proposition}
Let $(f_n)$ be a sequence of biholomorphisms from $D:=\{ |z| < 1 \}$ to $F_n := f_n(D)$, with $f_n(0) = 0$, $f'_n(0) > 0$. 
Then, $f_n$ converges locally uniformly in $D$   if and only if
 $(F_n)$ converges to its kernel $F$ and if $F \neq \C$. 
Moreover, the limit function maps  $D$ onto $F$. 
\end{proposition} 
From the Caratheodory theorem, it is possible to deduce the following property, which is crucial in our proof of Theorem \ref{theorem2} (notations are taken from the beginning of paragraph \ref{uniformestimates}):
\begin{proposition} 
Let $\Pi$ be the unbounded connected component of $\Omega$. There is a unique biholomorphism $\mathcal{T}$ from $\Pi$ to $\D$, satisfying $\mathcal{T}(\infty)=\infty$, $\mathcal{T}'(\infty) > 0$. Moreover, one has  the following convergence properties:  
\begin{itemize}
\item[i)] $\mathcal{T}_n^{-1}$ converges uniformly locally to $\mathcal{T}^{-1}$ in $\D$. 
\item[ii)] $\mathcal{T}_n$ (resp. $\mathcal{T}_n'$) converges uniformly locally to $\mathcal{T}$ (resp. to $\mathcal{T}'$) in $\Pi$. 
\item[iii)] $|\mathcal{T}_n|$ converges uniformly locally to $1$ in $\Omega\setminus\Pi$.
\end{itemize}
\end{proposition}
\begin{proof}
Let us first  point out that,  because of Hausdorff convergence and Proposition \ref{prop4}, any compact of $\Omega$ is included in $\Omega_n$ for $n$ large enough. Thus,  the local convergence properties stated in ii) and iii) make sense. 

\medskip
Up to a change of coordinates, we can always assume that $0 \in \pa \Omega \subset {\cal C}$. By Proposition \ref{prop3}, there exists $x_n \in \pa\Omega_n$ converging to $0$. 
Then, if we introduce the domains
$$ F_n \: := \: \frac{1}{(\Omega_n - x_n) \cup \{\infty\}} \:  := \: \left\{ \frac{1}{z}, \: z+x_n  \in \Omega_n  \right\} \cup \{ 0 \},$$
and 
$$  F \: := \: \frac{1}{\Pi\cup \{\infty\}} \: := \: \left\{ \frac{1}{z}, \: z  \in \Pi \right\}\cup \{ 0 \}, $$
it follows easily from (H1'), Proposition \ref{prop3} and Proposition \ref{prop4} that $F_n$ converges to $F$ in the kernel sense.   Note that by the choice of $(x_n)$, the $F_n$'s  do not include $\infty$.

\medskip
Hence, by the Caratheodory theorem, the sequence of biholomorphisms $(f_n)$ defined  by 
$$f_n : D \mapsto f_n(D) = F_n, \quad   f_n(z) \: := \: \frac{1}{\mathcal{T}_n^{-1}(1/z) - x_n} $$
converges uniformly locally in $D$ to some function $f$ from $D$ onto $F$. By the Weierstrass convergence theorem, $f$ is holomorphic over $D$. Moreover, by a standard application of the Rouch\'e formula, as $f_n$ is one-to-one for all $n$, so is $f$.  Going on with standard arguments, $f$ is the unique  biholomorphism that maps $D$ to $F$ and that satisfies  $f(0) = 0$, $f'(0) >0$.  Back to $\mathcal{T}_n^{-1}$, this yields i) with $\Tc^{-1}(z):=\dfrac{1}{f(1/z)}$. Actually, one has clearly  uniform convergence of $\mathcal{T}_n^{-1}$ to $\mathcal{T}^{-1}$ in $\Delta_{\delta} := \{ |z| \ge 1+\delta \}$ for all $\delta > 0$.   Then, by the Weierstrass theorem, the sequence of derivatives $(\mathcal{T}_n^{-1})'$ converges locally uniformly to $(\mathcal{T}^{-1})'$. 

\medskip
As regards ii), let  $z_0 \in \Pi$, and    $J_n :=   \mathcal{T}_n^{-1}(\{ z', | z' -  \mathcal{T}(z_0)|= \delta \})$. By i), for $\delta > 0$ small enough and $n$ large enough, $J_n$ is a closed curve that encloses $z_0$ and is contained in $\Pi$. For all $z$ in a small enough neighborhood of $z_0$, we can then write the Cauchy formula: 
$$  \mathcal{T}_n(z) \: = \: \frac{1}{2i\pi} \int_{J_n} \frac{ \mathcal{T}_n(\xi)}{\xi-z} d\xi   \: =  \:  \frac{1}{2i\pi} \int_{\{ |\xi'-\Tc(z_0)| = \delta\}} \frac{\xi'}{ \mathcal{T}_n^{-1}(\xi')-z} \,  (\mathcal{T}_n^{-1})'(\xi') \, d\xi' $$
where the last equality comes from the change of variable $\xi =  \mathcal{T}_n^{-1}(\xi')$. Thanks to i), we may let $n$ go to infinity to obtain the convergence of  $\mathcal{T}_n$ to $\mathcal{T}$ uniformly in a neighborhood of $z_0$. Again, the convergence of derivatives follows from the Weierstrass theorem. This ends the proof of ii). 

\medskip
To obtain iii), we argue by contradiction. We assume {\it a contrario} that there exists a $\delta  > 0$ and a sequence $z_n$ located in a given  closed ball  $B$  of $\Omega\setminus \Pi$ such that $|\mathcal{T}_n(z_n)| \ge 1 +\delta$. Up to extract a subsequence, we can assume that $z_n \rightarrow z \in B$. By the  uniform convergence of $\mathcal{T}_n^{-1}$ in   $\Delta_\delta$  (see above) we have  $z  \in \mathcal{T}^{-1}(\Delta_\delta) \subset \Pi$. Thus, we reach a contradiction, which proves iii).    

\end{proof}


\begin{thebibliography}{99}

\bibitem{Kondra} Borsuk M., Kondratiev V. {\it Elliptic boundary value problems of second order in piecewise smooth domains. } North-Holland Mathematical Library, 69. Elsevier Science B.V., Amsterdam, 2006. 

\bibitem{Bucur} Bucur D., Feireisl E., Necasova S. and Wolf J., On the asymptotic limit of the Navier-Stokes system with rough boundaries, {\it J. Diff. Equations},  244  (2008),  no. 11, 2890-2908. 

\bibitem{Simon} Casado-Diaz J., Fernandez-Cara E. and Simon J., Why viscous fluids adhere to rugose walls: a mathematical explanation, {\it J. Differential Equations},   189  (2003),  no. 2, 26-537. 

\bibitem{delort} Delort J.-M., Existence de nappes de tourbillon en dimension deux. (French) [Existence of vortex sheets in dimension two] {\it J. Amer. Math. Soc.} 4 (1991), no. 3, 553-586.

\bibitem{DiPernaMajda} DiPerna R. J. and Majda A. J.,  Concentrations in regularizations for 2-D incompressible flow, {\it  Comm. Pure Appl. Math.} 40 (1987), no. 3, 301-345. 



\bibitem{Galdi} Galdi G. P., {\it An introduction to the mathematical theory of the Navier-Stokes equations. Steady-state problems. Second edition}, Springer Monographs in Mathematics. Springer, New York, 2011.

\bibitem{henrot} Henrot A. and Pierre M., {\it Variation et optimisation de formes. Une analyse g\'eom\'etrique} (French) [Shape variation and optimization. A geometric analysis], Math\'ematiques $\&$ Applications 48, Springer, Berlin, 2005.

\bibitem{ift_lop_euler} Iftimie D., Lopes Filho M.C. and Nussenzveig Lopes H.J., Two Dimensional Incompressible Ideal Flow Around a Small Obstacle, {\it Comm. Partial Diff. Eqns.} 28 (2003), no. 1$\&$2, 349-379.

\bibitem{kiku} Kikuchi K.,  Exterior problem for the two-dimensional Euler equation, {\it J Fac Sci Univ Tokyo Sect 1A Math} 1983; 30(1):63-92.


\bibitem{lac_euler} Lacave C., Two Dimensional Incompressible Ideal Flow Around a Thin Obstacle Tending to a Curve, {\it Annales de l'IHP, Anl} \textbf{26} (2009), 1121-1148.

\bibitem{lac_small} Lacave C., Two-dimensional incompressible ideal flow around a small curve, {\it Comm. Partial Diff. Eqns.}, 37:4 (2012), 690-731.

\bibitem{lac_uni} Lacave C., Uniqueness for Two Dimensional Incompressible Ideal Flow on  Singular Domains, submitted, preprint 2011. {\tt arXiv:1109.1153}

\bibitem{Mazya} Kozlov V. A., Maz πya  V. G. and  Rossmann, J. {\it Spectral problems associated with corner singularities of solutions to elliptic equations. } Mathematical Surveys and Monographs, 85. American Mathematical Society, Providence, RI, 2001. 

\bibitem{milton} Lopes Filho M.C., Vortex dynamics in a two dimensional domain with holes and the small obstacle limit, {\it SIAM Journal on Mathematical Analysis}, 39(2)(2007) : 422-436.

\bibitem{majda} Majda A. and Bertozzi A.,  {\it Vorticity and Incompressible flow}, Cambridge University press, 2002. 

\bibitem{MG} McGrath F. J.,  Nonstationary plane flow of viscous and ideal fluids, {\it Arch. Rational Mech. Anal.} 27 1967 329–348.

\bibitem{pomm-1} Pommerenke C., {\it Univalent functions}, Vandenhoeck $\&$ Ruprecht, 1975.

\bibitem{Simon2} Simon J.,   Compact sets in the space Lp(0,T;B), {\it Ann. Mat. Pura Appl.} (4) 146 (1987), 65-96.

\bibitem{sverak} Sver\'ak V., On optimal shape design, {\it J. Math. Pures Appl.} (9) 72 (1993), no. 6, 537-551.

\bibitem{taylor} Taylor M., Incompressible fluid flows on rough domains. Semigroups of operators: theory and applications (Newport Beach, CA, 1998), 320-334, {\it Progr. Nonlinear Differential Equations Appl.}, 42, Birkhauser, Basel, 2000. 

\bibitem{Wolf} Wolf J., Existence of weak solutions to the equations of non-stationary motion of non-Newtonian fluids with shear rate dependent viscosity, {\it J. Math. Fluid Mech.} 9 (2007), no. 1, 104–138.

\bibitem{wolibner} Wolibner W. , Un th\'eor\`eme sur l'existence du mouvement plan  d'un fluide parfait homog\`ene, incompressible, pendant un temps infiniment long, {\it Math. Z.} \textbf{37} (1933), no. 1, 698-726.

\bibitem{yudo} Yudovich V. I., Non-stationary flows of an ideal incompressible fluid, {\it Z. Vycisl. Mat. i Mat. Fiz.} 3 (1963), pp. 1032-1066 (in Russian). English translation in {\it USSR Comput. Math. $\&$ Math. Physics} 3 (1963), pp. 1407--1456.

\end{thebibliography}
\end{document}